\documentclass[10pt,a4paper]{article}
\usepackage[latin1]{inputenc}
\usepackage{amsmath}
\usepackage{amsfonts}
\usepackage{amssymb}
\usepackage{color}
\newenvironment{proof}{\noindent
PROOF.}{\hfill$\mathbf{\Box}$ \vspace{.5\baselineskip}}
\DeclareMathOperator{\sgn}{sgn}

\numberwithin{equation}{section}
\title{Uniqueness for Volterra-type stochastic integral equations\\*
\protect\vspace{2cm}
}
\author{Leonid Mytnik  $\mbox{}^{1} $\hspace{2cm}  Thomas S. Salisbury $\mbox{}^{2}$}
\date{}
\begin{document}
\newcounter{archapter}[section]
\newtheorem{theorem}{Theorem}[section]
\newtheorem{lemma}[theorem]{Lemma}
\newtheorem{definition}[theorem]{Definition}
\newtheorem{remark}[theorem]{Remark}
\newtheorem{proposition}[theorem]{Proposition}
\newtheorem{case}{Case}[archapter]
\newtheorem{corollary}[theorem]{Corollary}
\newtheorem{fact}[theorem]{Fact}
\newtheorem{assumption}[theorem]{Assumption}
\newcommand{\gdm}{\hfill\vrule  height5pt width5pt \vspace{.1in}}
\newcommand{\ep}{\epsilon}
\newfont{\msb}{msbm10 scaled \magstep1}
\newfont{\msbh}{msbm7 scaled \magstep1}
\newfont{\msbhh}{msbm5 scaled \magstep1}
\newcommand{\IN}{\mathbb N}
\newcommand{\IP}{\mathbb P}
\newcommand{\IR}{\mathbb R}
\newcommand{\IE}{\mathbb E}
\renewcommand{\Re}{\mbox{\msb R}}
\newcommand{\Zd}{\mbox{\msb Z}}
\newcommand{\Ce}{\mbox{\msb C}}
\newcommand{\Cesm}{\mbox{\msbh C}}
\newcommand{\Te}{\mbox{\msb T}}
\newcommand{\Tesm}{\mbox{\msbh T}}
\newcommand{\Le}{\mbox{\msb L}}
\newcommand{\Lesm}{\mbox{\msbh L}}
\newcommand{\Ee}{\mbox{\msb E}}
\newcommand{\Pe}{\mbox{\msb P}}
\newcommand{\Resm}{\mbox{\msbh R}}
\newcommand{\Resmm}{\mbox{\msbhh R}}
\newcommand{\IZ}{\mathbb Z}
\newcommand{\Rp}{\Re_{+}}
\newcommand{\rmed}{\right|}
\newcommand{\lmed}{\left|}
\newcommand{\lnor}{\left\|}
\newcommand{\rnor}{\right\|}
\newcommand{\lnorm}{\lmed \lmed}
\newcommand{\rnorm}{\rmed \rmed}
\newcommand{\la}{\langle}
\newcommand{\ra}{\rangle}
\newcommand{\lla}{\left\langle}
\newcommand{\rra}{\right\rangle}
\newcommand{\Dd}{\Delta}
\newcommand{\Ddd}{\Delta_{d}}
\newcommand{\Ddl}{\frac{1}{2}\Dd}
\newcommand{\wX}{\widetilde{X}}
\newcommand{\mf}{\mathcal{M}_{\rm f}}
\newcommand{\ctem}{C_{\text{tem}}}
\newcommand{\crap}{C_{\text{rap}}}
\renewcommand{\theenumi}{\alph{enumi}}
\renewcommand{\labelenumi}{(\theenumi)}

\maketitle

\begin{abstract}
We study uniqueness for a class of Volterra-type stochastic integral equations. We focus on the case of non-Lipschitz noise coefficients. 
The connection of these equations to certain degenerate stochastic partial differential equations plays a key role.

\end{abstract}

\vspace*{6.5cm}

\thanks{\noindent 
\today\\
AMS 2000 {\it subject classifications}. 60H15, 60K35. 
 \\
{\it Keywords and phrases}. 
Strong uniqueness, stochastic partial differential equation, stochastic integral equation.\\
{\it Running head}. Uniqueness for Volterra-type equation\\
$\mbox{}^{1}$ Supported in part by  the Israel Science
 Foundation\\
$\mbox{}^{2}$ Supported in part by NSERC
 
}

\pagebreak
\section{Introduction}
\label{sec:intro}

The aim of this paper is to study uniqueness for a Volterra-type stochastic differential equation. Let $0< \alpha<1/2$
and let
$\sigma$ be a H\"older continuous function with exponent $\gamma\in (0,1)$. That is, we assume that there exists $L=L(\gamma)$ such that 
\begin{eqnarray}
\label{eq:sigma_Hold}
|\sigma(x)-\sigma(y)|\leq L|x-y|^{\gamma},\;\; \forall x,y\in \IR. 
\end{eqnarray}

Consider 
 the stochastic integral equation
\begin{eqnarray}
\label{equt:cat6}
  X_t = x_0 +\int_0^t(t-s)^{-\alpha} \mathsf{g}(s)\,ds+\int_0^t (t-s)^{-\alpha}\sigma(X_s) dB_s\,,\;\;t\geq0,
\end{eqnarray}
where $\mathsf{g}$ is a bounded continuous function. 
We will extend the classical
 Yamada-Watanabe strong uniqueness result~\cite{bib:YW71} to the above Volterra-type stochastic integral equation.
 
Existence and uniqueness for non-singular Lipschitz stochastic Volterra equations was shown in~\cite{bib:prot85}. To the best of our knowledge, there are no known
uniqueness results for these equations with non-Lipschitz coefficients. Indeed
a major difficulty 
lies in the absence of any natural semimartingale representation for solutions. However we
will show that some of the 
 methodology developed in~\cite{bib:msp05}, \cite{bib:mp} can be applied in this case.  

Here is the first main result of this paper.

\begin{theorem}
\label{thm:unique1}
Let $\alpha\in (0,1/2)$, and $\sigma$ satisfy~(\ref{eq:sigma_Hold}) for some $\gamma\in (\frac{1}{2(1-\alpha)},1]$. 
Then, for any $x_0\in \IR$,  and bounded continuous $\mathsf{g}$,  there is a pathwise unique solution to the 
 equation ~(\ref{equt:cat6}).  
\end{theorem} 
Note that Yamada-Watanabe result states uniqueness for the above equation 
in the case of $\alpha=0$ for any $\gamma\geq 1/2$. This gives an indication that our result is close to optimal, 
but  we have not succeeded in constructing a counterexample for the case of $\gamma<\frac{1}{2(1-\alpha)}$. 
We believe that the  methods developed in~\cite{bib:bmp} and \cite{bib:mmp} may be useful to tackle 
the non-uniqueness problem.

In fact our motivation for studying the above equation came originally from  the study of strong uniqueness for SPDEs and of catalytic superprocesses in dimension one. Recall that the density of super-Brownian motion with space-time dependent branching rate, in dimension $d=1$, can be represented as a solution to the following SPDE 
\begin{eqnarray} 
\label{equt:3} 
\frac{\partial X_{t}(x)}{\partial t}=\Ddl X_{t}(x)+
\sqrt{\lambda_s(x) X_{t}(x)}\dot{W},\;\;\;t\geq 0,\; x\in \IR. 
\end{eqnarray} 
 Here~$\lambda_s(x)$ maybe interpreted as an
 instantaneous  
rate of branching at the point $x$ at time $s$. If one takes $\lambda_s(x)=1$ and replaces the square root by a more general power, the SPDE 
\begin{eqnarray} 
\label{equt:classicalspde} 
\frac{\partial X_{t}(x)}{\partial t}=\Ddl X_{t}(x)+
|X_{t}(x)|^\gamma \dot{W},\;\;\;t\geq 0,\; x\in \IR. 
\end{eqnarray} 
has been studied extensively. \cite{bib:mp} shows strong uniqueness in \eqref{equt:classicalspde} for $\gamma>\frac34$,
 and \cite{bib:mmp} shows that strong uniqueness fails for $\gamma<\frac34$. 
 Those results are for unconstrained (ie signed) solutions. If solutions are restricted to be positive, as is the case for
 the density of super-Brownian motion, then \cite{bib:bmp} adds an immigration term and show that strong uniqueness
 fails for $\gamma<\frac12$. However altogether there is 
still no good understanding of uniqueness/non-uniqueness problem in  the range $\gamma\in(0,\frac34]$.

One may try to narrow that gap by considering a smoother process, namely {\it catalytic super-Brownian motion}. It will turn out that the analogue of the SPDE \eqref{equt:3} does not make sense, but that closely related SPDEs do, and lead naturally to the stochastic integral equation \eqref{equt:cat6}.

The study of super-Brownian motion with 
$\lambda_s(x)dx$ replaced by a singular measure $\rho_s(dx)$ was initiated 
in~\cite{bib:dawfle},~\cite{bib:dfr91},~\cite{bib:dawfle92}. 
  This pair $(\rho, 
 X)$ serves  as a model of a chemical (or biological) reaction 
 of two substances, called the {\it catalyst} and {\it reactant}. The 
 branching of the particles in the $X$ population ({\it reactant}) 
 occurs  only in the presence of {\em  catalyst} 
 $\rho$. 
 More specifically, $X$ is the super-Brownian motion 
 whose branching rate at time $t$ in the space element $dx$ is given by 
 $\rho_{t}(dx)$.  For $\rho$ the Dirac measure, an elegant  approach for studying the catalytic process was introduced 
 in~\cite{bib:fleleg95}. This approach was later extended  to a more general catalyst (see~\cite{bib:mv05}). 
The relation of catalytic super-Brownian motion to SPDEs was presented in~\cite{bib:z05}. However, in
the case where $\rho$ is the Dirac measure,  the catalytic SBM  cannot be rigorously described as a solution to an SPDE.
As we will see, there is a degenerate SPDE that is closely related to \eqref{equt:cat6}.

Let $\rho=\delta_0$. The process $X$ 
 makes non-trivial sense only in dimension
 $d=1$,  since only then do the paths of underlying Brownian particles hit the point
 catalyst. Before describing the corresponding martingale problem, it is 
 necessary to define the local time of a superprocess at point the $x=0$.   
 
At a heuristic level the 
 local time $l^0_{t}$ of a measure-valued process $X$ at the point $x=0$  
 is a non-decreasing real-valued process such that 
$$
l^0_{t}= \int_{0}^{t}\int_{\Resm}\delta_0(y)
 X_{s}(dy)\,ds,
$$
where $\delta_0$ is the Dirac delta 
 function. The precise definition includes using an approximate delta 
 function instead of $\delta_0$ and passing to the limit.
Then we can write the martingale problem for  super-Brownian motion with a point catalyst 
  at $x=0$ as 
$$
\label{equt:cat1}
\left\{
\begin{aligned}
 &X_{t}(\phi)=
 X_{0}(\phi)+\int_{0}^{t}
 X_{s}(\Delta\phi/2)\, ds+
 M_{t}(\phi), \forall \phi\in {\cal D}(\Delta),\\
&\text{where $M_{t}(\phi)$ 
 is a continuous square integrable ${\cal F}_{t}$-martingale} \\
&\text{with $\langle M(\phi)\rangle_{t}=
 \phi(0)^{2}  l^0(t)$ and  $M_{0}(\phi)=0$.}
\end{aligned}
\right.
$$

If we pretend that the measure $l^0(ds)$ is absolutely continuous with respect to Lebesgue 
measure, that is,
\begin{eqnarray}
\label{cat2}
 l^0_{t} = \int_{0}^{t} X_{s}(0)\,ds, 
\end{eqnarray}
and 
$X_{t}(0)$ is bounded,
then it would be easy to derive that $X_t(\cdot)$ is a solution to the following 
 degenerate SPDE written in
a  mild form:
\[ X_t(x) = \int_{\Resm} p_t(x-y)X_0(dy)
 +\int_0^t p_{t-s}(x)\sqrt{X_s(0)} dB_s\,,\]
where $p_t(x)$ is a transition density of Brownian motion (see~\cite{bib:z05} for 
  related results).  
Set  $x=0$ to get the following stochastic integral equation (SIE)
\begin{eqnarray}
\label{equt:cat4}
  X_t(0) = \int_{\Resm} p_t(y)X_0(dy) +\int_0^t \frac{1}{\sqrt{2\pi}}(t-s)^{-1/2}\sqrt{X_s(0)} dB_s\,.
\end{eqnarray}
However the assumption~(\ref{cat2}) is false -- the local time $l^0(ds)$ is singular 
 with respect to Lebesgue measure (see~\cite{bib:dawfle94},\cite{bib:dflm95}). 
In fact $X_t(dx)$ does not have a density at the point of catalyst $x=0$ 
and hence we do not expect there to be a solution to~(\ref{equt:cat4}) in the ordinary sense. But what we 
may ask about is the feasible parameters $\alpha$ and $\gamma$ such that there is a solution to the 
 following analogous SIE
\begin{eqnarray}
\label{equt:cat5}
  X_t(0) =  x_0 +\int_0^t (t-s)^{-\alpha}\lambda|X_s(0)|^\gamma dB_s\,,
\end{eqnarray}
where $\lambda\in\IR$. 
If we replace $|X_s(0)|^\gamma$ inside the stochastic integral, by a general $\gamma$-H\"older continuous function $\sigma(X_s(0))$ , we arrive at the equation~\eqref{equt:cat6}, uniqueness for which is the main concern of the
current paper. 

The connection between our SIE and SPDEs is more than simply an analogy or heuristic. In fact, some of
our arguments rely on rewriting the SIE in terms of an SPDE that can be thought of as the $\gamma>\frac12$ version of
a ``catalytic Bessel process''. In the particular case of $\gamma=1/2$, we can, in fact, 
show that there is at most one non-negative
solution of the equation~(\ref{equt:cat5}) {\it for any} $\alpha\in (0,1/2)$. Moreover, in our second main result, we
establish weak uniqueness for non-negative solutions of such equations. 
\begin{theorem}
\label{thm:weakUniq}
Assume that $\alpha\in (0,1/2)$, and $\lambda\in\IR$. 
Then for any $x_0>0$ and bounded non-negative continuous $\mathsf{g}$, there exists at most one weak non-negative solution to 
\begin{eqnarray}
\label{equt:cat_pos}
  X_t=  x_0 +\int_0^t(t-s)^{-\alpha} \mathsf{g}(s)\,ds+\int_0^t  (t-s)^{-\alpha}\lambda\sqrt{|X_s|} dB_s\,,\;\;t\geq 0. 
\end{eqnarray}
\end{theorem}

\paragraph{Organization of the paper} In Section~\ref{sec:2}, we first 
prove existence results for our equations. Then,
 in Proposition~\ref{prop:2}, we treat the case of $\gamma=1$ of Theorem~\ref{thm:unique1}.
 In the same section, we introduce the SPDE analogues of equations~\eqref{equt:cat6} and \eqref{equt:cat_pos}, and 
  state 
 corersponding uniqueness Theorems~\ref{thm:uniquekunbounded}, \ref{thm:unique_weak_SPDE}. In 
  Section~\ref{section:weak_uniqueness}, Theorems~\ref{thm:unique_weak_SPDE} and~\ref{thm:weakUniq} are proved. 
 The rest of the paper, except Section~\ref{sec:smoothkernels} is devoted to the proof of 
 Theorem~\ref{thm:uniquekunbounded}, 
 from which Theorem~\ref{thm:unique1}, is an immediate consequence. In Section~\ref{sec:smoothkernels}, the 
  uniqueness for equations with kernels smoother than $(t-s)^{-\alpha}$ is considered.

\section{Existence and background} 
\label{sec:2}
Our first goal is to construct the solution to~(\ref{equt:cat6}).
In fact, we will prove existence of a solution to the more 
general equation:
\begin{eqnarray}
\label{equt:21_1}
  X_t =  h(t) +\int_0^t (t-s)^{-\alpha}\sigma(X_s) \,dB_s\,,
\end{eqnarray}
where $h$ is a continuous function. 

Before we start dealing with the above questions, we introduce some notation, which will
be used throughout this work.  We write $C(\IR)$ for the space of continuous
functions on $\IR.$ A superscript $k$ (respectively $\infty$) indicates that
functions are in addition $k$ times (respectively infinitely many times) continuously
differentiable. A subscript $b$ (respectively $c$) indicates that they are also
bounded (respectively have compact support). We also define tempered norms
\begin{equation*}
||f||_{\lambda,\infty}:=\sup_{x \in \IR} |f(x)| e^{-\lambda |x|},
\end{equation*}
set $$\ctem:=\{f \in C(\IR) , ||f||_{\lambda,\infty}
< \infty  \text{ for every } \lambda >0\}$$ and endow it with the topology induced
by the norms $||\cdot ||_{\lambda,\infty}$ for $\lambda>0.$ That is,
$f_n\to f$ in $\ctem$ iff $\lim_{n\to\infty}\Vert
f-f_n\Vert_{\lambda,\infty}=0$ for all $\lambda>0$. 
Similarly we define  $$\crap:=\{f \in C(\IR) , ||f||_{\lambda,\infty}
< \infty  \text{ for every } \lambda <0\}$$ and endow it with the topology induced
by the norms $||\cdot ||_{\lambda,\infty}$ for $\lambda<0.$
$\ctem^k$ (respectively $\crap^k$) denotes collection of functions in  $\ctem$ (respectively in $\crap$) which  
are in addition $k$ times continuously
differentiable with all the derivatives in  $\ctem$ (respectively in $\crap$). As before $k$ can be equal to $\infty$. 

For
$I\subset \IR_+$, let
$C(I,E)$ be the space of all continuous functions on $I$ taking values in
a topological space $E$, endowed with the topology of uniform convergence
on compact subsets of $I$.  In particular, $X\in C(\IR_+,\ctem)$ denotes a function $X_t(x)$ with $X_t\in \ctem$ varying continuously
with $t$. In this context we will use either the notation $X(t,x)$ or $X_t(x)$, depending on which is more convenient.
We will also denote by~$\ctem^+$ the collection of non-negative functions in $\ctem$. Let 
$\mf=\mf(\IR)$ be the space of finite measures on $\IR$ endowed with weak topology. Throughout the paper
$c_i$ and $c_{i.j}$  will denote fixed positive constants, while $C$ and $c$ will
denote positive constants which may change from line to line.

Now we return to the equation~\eqref{equt:21_1}
First, let us treat the case of Lipschitz $\sigma$  (by this, we  prove  Theorem~\ref{thm:unique1} for the case of $\gamma=1$): 
\begin{proposition}
\label{prop:2}
Let $\sigma$ be a
continuous Lipschitz function. 
Assume $0<\alpha<\frac12$ and $h\in C(\IR_+,\IR)$.
Then there exists a unique strong solution  $X$ to~\eqref{equt:21_1} in $C(\IR_+,\IR)$. Moreover, for any $p>0$, $T>0$, there exists 
a constant $c_{\ref{equt:22_3}}=c_{\ref{equt:22_3}}(p,T,\sigma)<\infty$ such that 
\begin{eqnarray}
\label{equt:22_3}
\sup_{0\leq s\leq T} \IE\left[ \left|X_s\right|^p\right] 
&<& c_{\ref{equt:22_3}}(p,T,\sigma).   
\end{eqnarray}
\end{proposition}
\begin{proof}
We will use the standard Picard scheme. Let 
\begin{eqnarray}
\nonumber
X^0_t&=& h(t) \\
\label{equt:21_3}
X^{n+1}_t&=& h(t) + \int_0^t (t-s)^{-\alpha}\sigma(X^n_s) \,dB_s,\;\; n\geq 0.  
\end{eqnarray}
Note that since $\sigma$ is Lipschitz it also satisfies a 
linear growth bound,  
that is,  
there exists a constant $c_{\ref{growthcond}}$ such that
\begin{equation}
\index{linear growth condition}
\label{growthcond}
|\sigma(x)| \leq c_{\ref{growthcond}} (1 + |x|),\;\;\;\forall x\in\IR.
\end{equation}
First, let us prove by induction that $X^n$ is well defined for all $n$. 
In what follows, we fix an arbitrary $T>0$. Assume inductively that $X^n$ is a well defined 
adapted process and
\begin{eqnarray}
\label{equt:21_2}
\sup_{0\leq t\leq T} \IE\left[ \left|X^n_t\right|^2\right] <\infty. 
\end{eqnarray}
Then, by  using the growth  condition~\eqref{growthcond}, one can immediately get 
that the stochastic integral in~\eqref{equt:21_3} is well defined and hence $X^{n+1}$ is well defined, and 
moreover,
\begin{eqnarray}
\nonumber
\sup_{0\leq t\leq T} \IE\left[ \left|X^{n+1}_t\right|^2\right]&\leq& 
2\sup_{0\leq t\leq T}h(t)^2 + \frac{4}{1-2\alpha}T^{1-2\alpha} c_{\ref{growthcond}}^2
 \sup_{0\leq t\leq T}\left(1+\IE\left[ \left|X^n_t\right|^2\right]\right)
    \\
\label{equt:21_4}
&<&\infty.   
\end{eqnarray}
So, by induction, we immediately get that $X^n$ is well defined for all $n$, and~\eqref{equt:21_2} holds for all $n$. 
Similarly, by using Burkholder-Gundy-Davis and H\"older inequalities, one can show that, for any $p\geq 2$, there exists 
a constant $c_{\ref{equt:21_5}}=c_{\ref{equt:21_5}}(p, c_{\ref{growthcond}})<\infty$
such that
\begin{eqnarray*}
\nonumber
\sup_{0\leq s\leq t} \IE\left[ \left|X^{n+1}_s\right|^p\right]&\leq& 
c_{\ref{equt:21_5}} \left(
\sup_{0\leq s\leq T}h(t)^p \right. \\
\nonumber
&&\left. \mbox{} + \int_0^t (t-s)^{-2\alpha}
 \left(1+\sup_{0\leq u\leq s}\IE\left[ \left|X^n_u\right|^p\right]\right)\,ds\right)
    \\
\label{equt:21_5}
&<&\infty.   
\end{eqnarray*}
By this, and by the extension of Gronwall's lemma (see Lemma~15 in~\cite{bib:dal99}), we get that, in fact, there exists 
a constant $c_{\ref{equt:21_6}}=c(p, c_{\ref{growthcond}},T)<\infty$ such that  
\begin{eqnarray}
\label{equt:21_6}
\sup_{n\geq 0}\sup_{0\leq s\leq T} \IE\left[ \left|X^{n}_s\right|^p\right] &\leq & c_{\ref{equt:21_6}}.   
\end{eqnarray}
Now in order to show that the sequence $\{X^n_t\}_{n\geq 0}$ converges in $L^p$, define
$$V^n_t=  \sup_{0\leq s\leq t} \IE\left[ \left|X^{n+1}_s-X^n_s\right|^p\right]. $$
Since $\sigma$ is Lipschitz function, we conclude similarly to~\eqref{equt:21_5} that there exists a constant 
$c_{\ref{equt:22_1}}=c_{\ref{equt:22_1}}(p,c_{\ref{growthcond}})<\infty$ such that
\begin{eqnarray}
V^{n+1}_t&\leq& 
c_{p,\ref{equt:22_1}} \int_0^t (t-s)^{-2\alpha}
V^n_s\,ds. 
\label{equt:22_1}   
\end{eqnarray}
Since, by \eqref{equt:21_6}, $\sup_{0\leq t\leq T} V^0_t<\infty$, we again 
get by the extension of Gronwall's lemma (see Lemma~15 in~\cite{bib:dal99}) that $\{V^n_t\}_{n\geq 0}$ converges to $0$ 
uniformly on $[0,T]$. This inmplies that there exists $X_t$ such that $\{X^n_t\}_{n\geq 0}$ converges to $X$ in $L^p$ 
uniformly on $[0,T]$. It is easy to check that $X$ has a jointly measurable version, and that $X$, 
in fact,  satisfies~\eqref{equt:21_1} for a.e. $t$. The existence of continuous in time version 
of the process follows by standard application of Kolmogorv continuity 
criterion and is left to the reader. As we choose the continuous version of the process we get that $X$ satisfies 
 \eqref{equt:21_1} for all $t$. 

To prove uniqueness, let $X^1_t$ and $X^2_t$ solve~\eqref{equt:21_1}. Suppose $|\sigma(x)-\sigma(y)|\le c|x-y|$. For $K>0$, let $T_K$ be the first time $t$ that either of $|X^i_t|>K$. Set 
$$
m_K(t)=\sup_{s\le t}\IE(X^1_{s\land T_K}-X^2_{s\land T_K})^2]\le 4K^2<\infty.
$$
Then for $s\le t$, 
\begin{align*}
\IE[(X^1_{s\land T_K}-X^2_{s\land T_K})^2]
&= \IE[\int_0^{s\land T_K} (s-q)^{-2\alpha}(\sigma(X^1_{q})-\sigma(X^2_{q}))^2\,dq]\\
&\le c\IE[\int_0^{s} (s-q)^{-2\alpha}(X^1_{q\land T_K}-X^2_{q\land T_K})^2\,dq]\\
&\le cm_K(t)\int_0^{s} (s-q)^{-2\alpha}\,dq=cm_K(t)\frac{t^{1-2\alpha}}{1-2\alpha}.
\end{align*}
Therefore $m_K(t)\le cm_K(t)\frac{t^{1-2\alpha}}{1-2\alpha}$, from which we conclude that $m_K(t)=0$ on some interval $[0,\epsilon]$. Iterating the argument now shows that $m_K(t)=0$ for all $t\ge 0$, and sending $K\to\infty$ implies the desired result.

As for \eqref{equt:22_3}, it  follows immediately by~\eqref{equt:21_6}.
\end{proof}
\begin{remark}
\label{rem:22_1}
Note that the constant $c_{\ref{equt:22_3}}$ depends on $\sigma$ only through the constant $c_{\ref{growthcond}}$. 
\end{remark}

We now turn to the non-Lipschitz case.
\begin{lemma}
\label{cor:wkexistenceforSIE}
Let $\sigma$ be
continuous and satisfy the growth bound
(\ref{growthcond}). Assume $0<\alpha<\frac12$ and  and $h\in C(\IR_+,\IR)$. Then there exists a weak solution 
$X$ to~\eqref{equt:21_1} in $C(\IR_+,\IR)$ and 
\begin{eqnarray}
\label{equt:23_3}
\sup_{0\leq s\leq T} \IE\left[ \left|X_s\right|^p\right] 
&<& c_{\ref{equt:23_3}}(p,T,\sigma).   
\end{eqnarray} 
\end{lemma}
\begin{proof}
Choose  a sequence of Lipschitz functions $\{\sigma_n\}_{n\geq 1}$ which satisfy the growth condition~\eqref{growthcond} 
uniformly in $n$, and such that $\{\sigma_n\}_{n\geq 1}$ converges to $\sigma$ uniformly on $\IR$, as $n\rightarrow \infty$. 
Then by the previous proposition for each $n\geq 1$ there exists a  process $X^n$ 
that solves  \eqref{equt:21_1} with $\sigma_n$. 
Since $\sigma_n$ satisfy the growth condition~\eqref{growthcond} with the same constant, by
\eqref{equt:21_6} and Remark~\ref{rem:22_1}, we get that for any $T>0, p\geq 2,$
\begin{eqnarray}
\label{equt:22_4}
\sup_{n\geq 1}\sup_{0\leq s\leq T} \IE\left[ \left|X^{n}_s\right|^p\right] &
< \infty.   
\end{eqnarray}
Now, for any $0\leq t<t'$,
we have
\begin{eqnarray*}
|X^n_{t'}-X^n_t|^p&\leq& 
C_p|h(t')-h(t)|+ 
 C_p\int_0^t ((t'-s)^{-\alpha}- (t-s)^{-\alpha})\sigma_n(X^n_s)\,dB_s\\
&&\mbox{}+ C_p\int_t^{t'} (t'-s)^{-\alpha} \sigma_n(X^n_s)\,dB_s\,.
\end{eqnarray*}
To bound the expectations of the three terms on the right hand side, use  
the Burkholder-Davis-Gundy and H\"older inequalities,  \eqref{growthcond}, \eqref{equt:22_4} and some simple algebra. This implies 
\[ \IE\left[|X^n_{t'}-X^n_t|^p\right]\leq C |t'-t|^{p(1/2-\alpha)},\]
where the constant on the right hand side does not depend on $n$. 
By the Kolmogorov criterion we get the tightness of $\{X^n\}_{n\geq 1}$ in $C(\IR_+, \IR)$, and each 
 weak limit point 
is H\"older contiuous with any index less than $1/2 -\alpha$. 

Let $\{X^{n_k}\}_{k\geq 1}$ be some converging subsequence and $X$ be the corresponding limit point. First, clearly
\eqref{equt:23_3} follows from~\eqref{equt:22_4}. 
We will show that $X$ satisfies~\eqref{equt:21_1}. 
Define
\begin{eqnarray*}
Y^k_t&=& \int_0^t(t-s)^{\alpha-1} X^{n_k}_s\,ds,\;\;t\geq 0,\;k\geq 1. 
\end{eqnarray*}
It is easy to check that $Y^k$ satisfies the following equation 
\begin{eqnarray*}
Y^k_t&=&  \int_0^t(t-s)^{\alpha-1} h(s)\,ds+ 
 c_{\alpha}\int_0^t \sigma_{n_k}(X^{n_k}_s)\,dB^{n_k}_s,
\end{eqnarray*}
where $c_{\alpha}=\int_0^1 (1-r)^{\alpha-1}r^{-\alpha}\,dr$. By passing to the limit, due to convergence of 
$\{X^{n_k}\}_{k\geq 1}$ 
in $C(\IR_+, \IR)$ we get that $\{Y^k\}_{k\geq 1}$ converges in $C(\IR_+, \IR)$ to 
$Y_t=\int_0^t(t-s)^{\alpha-1} X_s\,ds,\;\;t\geq 0.$ Moreover 
$$M^k_t\equiv c_{\alpha}\int_0^t \sigma_{n_k}(X^{n_k}_s)\,dB^{n_k}_s\,, t\geq 0,\;k\geq 1,$$ is a sequence of 
square integrable martingales with quadratic variations given by 
$$\langle M^k_{\cdot}\rangle_t =  c_{\alpha}^2\int_0^t \sigma_{n_k}(X^{n_k}_s)^2\,ds\,, t\geq 0,\;k\geq 1.$$
By the  uniform integrability, uniform convergence of $\{\sigma_{n_k}\}_{k\geq 1}$ to $\sigma$ 
 and again by convergence of $\{X^{n_k}\}_{k\geq 1}$ 
 in $C(\IR_+,\IR)$, we get that martingales converge to the martingale $M$ with quadratic varation 
$$\langle M_{\cdot}\rangle_t =  c_{\alpha}^2\int_0^t \sigma(X_s)^2\,ds\,, t\geq 0.$$
Now it is standard to show that there exists a Brownian motion $B$ such  that
$M_t=  c_{\alpha}\int_0^t \sigma(X_s)\,dB_s,\; t\geq 0$, and hence,
\begin{eqnarray*}
Y_t&=&  \int_0^t(t-s)^{\alpha-1} h(s)\,ds+ 
 c_{\alpha}\int_0^t \sigma(X_s)\,dB_s,\;t\geq 0.
\end{eqnarray*}
By reversing the transformation, that is, by recalling that 
 $$X_t=\frac{1}{c_{\alpha}}\frac{d}{dt} \int_0^t (t-s)^{-\alpha}Y_s\,ds,\;t\geq 0,$$ it is easy to verify  that $X$ is a solution 
 to~\eqref{equt:21_1}.
\end{proof}

We will now construct an SPDE related~\eqref{equt:21_1}. 
Fix $\theta >0$. Define
\begin{eqnarray}
\Delta_\theta&=& \frac{2}{(2+\theta)^2}\frac{\partial }{\partial x} \left| x\right|^{-\theta}
  \frac{\partial }{\partial x}.
\end{eqnarray}
Then, for some constant $c_{\theta}>0$,  the function
\begin{eqnarray}
\label{equt:32}
 p^{\theta}_t(x) =\frac{c_{\theta}}{t^{\frac{1}{2+\theta}}}e^{-\frac{|x|^{2+\theta}}{2t}}
\end{eqnarray}
is a classical solution to the following evolution equation 
\begin{eqnarray*}
\left\{\begin{array}{rcl}
\frac{\partial u}{\partial t}&=&\Delta_\theta u \\
  u_0&=& \delta_0
\end{array}
\right. 
\end{eqnarray*} 
on $\IR_+\times\IR$. 
By changing variables, we see that $\int p^\theta_t(x)\,dx$ is independent of $t$, and we choose $c_\theta$ to make $p^\theta_t$ a probability density. Note that $\Delta_0=\frac12\Delta$, where $\Delta$ is the classical Laplacian. 


Let $\{S_t\,, t\geq 0\}$  be the 
semigroup generated by $\Delta_{\theta}$. That is, 
\begin{eqnarray}
S_t\phi(x)=\int_{\IR}p^{\theta}_t(x,y)\phi(y)\,dy,
\end{eqnarray}
where $p^{\theta}_t(x,y)$ is the transition density for the process with generator $\Delta_\theta$ (ie the fundamental solution to $\dot u=\Delta_\theta u$).
Define the domain of the operator $\Delta_{\theta}$:
\begin{eqnarray}
D(\Delta_{\theta}) \equiv \left\{ \phi\in \crap^2: \Delta_{\theta}\phi\in \crap\right\}.
\end{eqnarray}
In certain cases we will need also domain containing more functions:
\begin{eqnarray}
D_{\rm tem}(\Delta_{\theta})\equiv \left\{ \phi\in \ctem^2: \Delta_{\theta}\phi\in \ctem\right\}.
\end{eqnarray}

The generator is ambiguous at $x=0$, but we choose the semigroup to be symmetric; $p_t(0,x)=p_t(0,-x)$. Because $\Delta_\theta$ is in divergence form, $p^\theta_t(x,y)=p^\theta_t(y,x)$, so in particular,
$$
p_t^\theta(x,0)=p_t^\theta(0,x)=p_t^\theta(x),
$$
where the latter is given by \eqref{equt:32}. It is simple to verify that if a process $\xi_t$ has semigroup $S_t$ then 
$|\xi_t|^{1+\frac{\theta}{2}}\text{sign}(\xi_t)$ is a Bessel process of dimension $\frac{2}{2+\theta}<1$, so in fact, explicit formulas for $p_t(x,y)$ could be given.


\begin{lemma}
\label{thm:nonLipschitzexistence}
Let $X_0\in \ctem$ and $g\in C(\IR_+\,,\ctem)\cup C(\IR_+\,,\mf)$. 
Assume that $\theta\in(0,\infty)$ and that $\sigma$ is continuous and satisfies a 
linear growth condition \eqref{growthcond}. 
Then there exists a weak solution $X\in C(\IR_{+}\,,\ctem)$  to the following SPDE
\begin{eqnarray}
\label{equt:4}
X(t,x)&=& S_t X_0(x)+
\int_0^t p^{\theta}_{t-s}(x)\sigma(X(s,0))
 \,dB_s\\
\nonumber
&&\mbox{}+ \int_0^t \int_{\Resm}p^{\theta}_{t-s}(x,y) g(s,y)\,dy\,ds,\;t\geq 0. 
 \end{eqnarray} 

\end{lemma} 
\begin{proof}
By Lemma~\ref{cor:wkexistenceforSIE} there exists  weak solution $V$ to the SIE
\begin{eqnarray*}
  V_t =  h(t) +\int_0^t c_{\theta} (t-s)^{-\alpha}\sigma(V_s) dB_s\,,
\end{eqnarray*}
with 
$$ h(t) =S_tX_0(0) + \int_0^t \int_{\Resm}p^{\theta}_{t-s}(0,y) g(s,y)\,dy\,ds, 
$$
$c_{\theta}$ as in\eqref{equt:32}, and 
\begin{eqnarray}
\label{equt:22_5}
\alpha\equiv \frac{1}{2+\theta}.
\end{eqnarray}
Now define 
\begin{eqnarray*}
X(t,x)= S_tX_0(x) + \int_0^t \int_{\Resm}p^{\theta}_{t-s}(x,y) g(s,y)\,dy\,ds
 +\int_0^t p^{\theta}_{t-s}(x)\sigma(V_s)
 \,dB_s
\end{eqnarray*}
It is trivial to check that $X$ is indeed solution to~\eqref{equt:4} with $X(t,0)=V_t\,, t\geq 0,$
and $X$ is in 
$ C(\IR_{+}\,,\ctem)$

\end{proof}


 For the rest of the 
paper we will also assume~\eqref{equt:22_5}. 

It is clear from Lemma~\ref{thm:nonLipschitzexistence} and its proof that there is a correspondence between SPDEs of type~\eqref{equt:4} and SIEs of 
type~\eqref{equt:21_1}.
Consider the particular case with $X_0=x_0={\rm const}$ and
$g(s,x)=\frac{1}{c_{\theta}} \mathsf{g}(s)\delta_0$ in~\eqref{equt:4}.
  Then $S_tx_0=x_0$, so~\eqref{equt:4} becomes
\begin{eqnarray}
\nonumber
X(t,x)&=&
x_0+ \int_0^t p^{\theta}_{t-s}(x)  \frac{\mathsf{g}(s)}{c_{\theta}}\,ds+
\int_0^t p^{\theta}_{t-s}(x)\sigma(X(t,0))
\,dB_s.
 \end{eqnarray}
In particular for $x=0$ we have 
\begin{eqnarray}
\nonumber
X(t,0)&=&x_0+ \int_0^t (t-s)^{-\alpha}  \mathsf{g}(s)\,ds+
\int_0^t  c_\theta (t-s)^{-\alpha}\sigma(X(s,0))
 \,dB_s.
\end{eqnarray}
Thus we get that $X_t=X(t,0)$ 
satisfies the SIE given in~\eqref{equt:cat6} with $c_\theta\sigma(\cdot)$ instead of 
 $\sigma(\cdot)$. 
Conversely, if $X_t$ is a solution to~\eqref{equt:cat6} ith $c_\theta\sigma(\cdot)$ instead of 
 $\sigma(\cdot)$, then as in the proof of Lemma~\ref{thm:nonLipschitzexistence}
we can define
 \begin{eqnarray*}
\label{equt:5}
X(t,x)&\equiv&x_0+ \int_0^t p^{\theta}_{t-s}(x)  \frac{\mathsf{g}(s)}{c_{\theta}}\,ds+
\int_0^t  p^{\theta}_{t-s}(x)
\sigma(X(s))\,dB_s. 
\end{eqnarray*}
Then  $X(\cdot,\cdot)$ lies in $C(\IR_{+}\,,\ctem)$ and satisfies~(\ref{equt:4}) with $X(0,\cdot)=x_0$ and 
$g(s,\cdot)=\frac{\mathsf{g}(s)}{c_{\theta}}\delta_0(\cdot)$. 
Thus Theorem~\ref{thm:unique1} will follow if we can show pathwise uniqueness for~(\ref{equt:4}). In order to prove  the pathwise uniqueness for (\ref{equt:cat6}) it is enough to prove the
pathwise uniqueness for~(\ref{equt:4}). In other words, it follows once we prove the following theorem. 
\begin{theorem}
\label{thm:uniquekunbounded}
Assume that $\alpha\in (0,1/2)$ and that $\sigma:\IR\to\IR$ satisfies 
(\ref{growthcond}) and~(\ref{eq:sigma_Hold}) 
for some
$\gamma\in
(\frac{1}{2(1-\alpha)} , 1]$. 
Let $X_0\in \ctem$ and $g\in C(\IR_+\,,\ctem)\cup C(\IR_+\,,\mf)$.
Then pathwise uniqueness holds for 
solutions of  (\ref{equt:4}) in $C(\IR_{+},\ctem)$.
\end{theorem}

\bigskip 
The proof of our pathwise uniqueness theorems will require some moment bounds and regularity 
properties 
for {\it arbitrary} continuous 
$\ctem$-valued solutions to the equation~(\ref{equt:4}).
We know that the fractional Brownian motion $\int_0^t (t-s)^{-\alpha}\,dB_s$ is 
H\"older continuous with exponent $\xi$ for any  \mbox{$\xi<\frac12-\alpha$}. 
More generally, if $X$ is any solution to ~(\ref{equt:4}) then $X(t,0)$  is H\"older  with exponent $\xi$ for any  \mbox{$\xi<\frac12-\alpha$}. 
In fact, we have the following result.

\begin{proposition}
\label{prop}
Let $X_{0} \in \ctem\,, g\in C(\IR_+\,,\ctem)\cup C(\IR_+\,,\mf),\alpha\in (0,1/2)$ and  let $\sigma$ be a
continuous function satisfying the growth bound~(\ref{growthcond}).
 Then any solution
$X\in C(\IR_{+}, \ctem)$ to (\ref{equt:4})
has  the following properties.
\begin{itemize}
\item[{\bf (a)}]
For any $T,\lambda >0$ and $p \in (0,\infty),$
\begin{equation}
\label{eq:uniformmoments2a}
\sup_{0\leq t\leq T}\sup_{x\in\IR}\IE\Big( 
|X(t,x)|^{p} e^{-\lambda |x|}   \Big)
< \infty.
\end{equation}
\item[{\bf (b)}]
For any   $\xi \in (0, \frac12-\alpha)$ the process $X(\cdot,\cdot)$ is a.s.
uniformly H\"older continuous on compacts in $(0,\infty)\times [-1,1]$, and 
the process $$Z(t,x)\equiv X(t,x)-S_tX_0(x)-\int_0^t \int_{\Resm}p^{\theta}_{t-s}(x,y) g(s,y)\,dy\,ds,
$$ is uniformly H\"older continuous on
compacts in $[0,\infty)\times [-1,1]$, with H\"older coefficients
$\xi$ in time and space.

Moreover, for any $T>0$, $R>0$, and
$0\le t,t' \leq T, x,x' \in \IR$ such that $|x|,|x'|\leq 1$ as well as $p \in [2,
\infty)$,
there exists a constant
$c_{\ref{eq:Kolmogorov}}=c_{\ref{eq:Kolmogorov}}(T,p)$ such that
\begin{equation}
\label{eq:Kolmogorov}
\IE\left( |Z(t,x) - Z(t',x') |^{p}\right)
\leq  c_{\ref{eq:Kolmogorov}}\left(|t-t'|^{(1/2-\alpha) p}
+ |x-x'|^{(1/2-\alpha)p} \right).
\end{equation}
\end{itemize}
\end{proposition}
The proof of Proposition~\ref{prop} is delayed to Section~\ref{sec:prop}.

It is straightforward to show that under the hypotheses of Lemma
\ref{thm:nonLipschitzexistence}, solutions to (\ref{equt:4}) with
continuous $\ctem$-valued paths are also solutions to the equation in its
distributional form for suitable test functions
$\Phi$.
More specifically, for $\Phi \in C(\IR_+, D(\Delta_{\theta})$, such that $s\mapsto\frac{\partial \Phi_s(\cdot)}{\partial s} \in 
 C(\IR_+, \crap)$, we have
\begin{eqnarray}
\label{weakheatb}
\int_{\IR} X(t,x) \Phi_t(x)dx
&=& \int_{\IR} X_{0}(x)  \Phi_0(x)dx
\\
\nonumber
& &
+\int_{0}^{t}\int_{\IR} X(s,x) \left(\Delta_{\theta}\Phi_s(x)+ \frac{\partial \Phi_s(x)}{\partial s}\right) dx ds
\\
\nonumber
& &+\int_{0}^{t}\int_{\IR}g(s,x)\Phi_s(x)\,dx\, ds
\\
\nonumber
& &+\int_{0}^{t}\sigma(X(s,0))\Phi_s(0) \,dB_s
\quad\forall t\ge 0 \quad a.s.
\end{eqnarray}
In fact, given an appropriate class of test functions,
the two notions of solution (\ref{equt:4}) and (\ref{weakheatb})
are equivalent. 
For the
details of  a similar proof we refer to Shiga~\cite{bib:shiga94} Theorem~2.1 and its proof.
There, 
the setting is a bit different as it works in the setting 
of a non-degenerate SPDE. However, the
arguments do not change as long as the stochastic
integral in (\ref{weakheatb}) is well defined, which can easily be checked.

\medskip

Now we say a few words about the proof of Theorem~\ref{thm:weakUniq}. Again, by
Lemma~\ref{thm:nonLipschitzexistence}, its proof and discussion after it, it is enough 
to prove the weak uniqueness for corresponding SPDE. That is we are going to prove the following resut. 

\begin{theorem}
\label{thm:unique_weak_SPDE}
Assume that $\alpha\in (0,1/2)$,  and $\lambda\in\IR$.
Let $X_0\in \ctem^+$ and\\  \mbox{$g\in C(\IR_+\,,\ctem^+)\cup C(\IR_+\,,\mf)$.} 
Then there exists at most one weak solution $X\in C(\IR_{+}\,,\ctem^+)$  to the following SPDE
\begin{eqnarray}
\label{equt:4_weak}
X(t,x)&=& S_t X_0(x)+
\int_0^t p^{\theta}_{t-s}(x)\lambda \sqrt{X(s,0)}
 \,dB_s\\
\nonumber
&&\mbox{}+ \int_0^t \int_{\Resm}p^{\theta}_{t-s}(x,y) g(s,y)\,dy\,ds,\;t\geq 0. 
 \end{eqnarray} 
\end{theorem}

\bigskip

\section{Proof of Theorems~\ref{thm:unique_weak_SPDE} and~\ref{thm:weakUniq}}
\label{section:weak_uniqueness}
We start with proving Theorem~\ref{thm:unique_weak_SPDE}. By simple scaling, we may and will assume, 
without loss of generality, that $\lambda =1$. 

We need the following lemma. 
\begin{lemma}
Let $\phi$ be a non-negative function in  $C_c(\IR)$. Then there exists a unique, non-negative solution 
$u=U^{\phi}\in C(\IR_+, \crap(\IR))$ to the following equation
\begin{eqnarray}
\label{eq:mildLL}
u(t,x)&=& S_t\phi(x) - \int_0^t p^{\theta}_{t-s}(x)\frac{1}{2}u(s,0)^2\,ds,\;\;t\geq 0, \; x\in \IR.   
\end{eqnarray}
\end{lemma}
\begin{proof}
The proof is an easy adaptation of the proof of Proposition 2.3.1 in~\cite{bib:dawfle94}  for our cituation 
where the Brownian semigroup and its kernel is replaced by the semigroup and the kernel generated by $\Delta_{\theta}$. 
It is not difficult  to see that all the basic estimates hold also in this case. Note that the proof of 
Proposition 2.3.1 in~\cite{bib:dawfle94} also uses the ideas from 
the proof of 
Theorem 3.5 in~\cite{bib:dfr91} where a more general set of ``catalysts'' is considered.
The fact that for any $t\geq 0$, $u(t,\cdot)\in \crap(\IR)$ is an easy consequence of the domination 
$$ u(t,x)\leq S_t\phi(x),\;\; t\geq 0, x\in \IR. $$
\end{proof}

For any two functions $\phi, \psi$ on $\IR$, denote 
$$ \la \phi,\psi\ra\equiv \int_{\Resm} \phi(x)\psi(x) dx$$
whenever integral exists. 

Now we need the following lemma. 
\begin{lemma}
\label{lem:uniq_onedim}
Let $X_0\in \ctem^+$ and \mbox{$g\in C(\IR_+\,,\ctem^+)\cup C(\IR_+\,,\mf)$.} 
Let  $X$ be any solution to (\ref{equt:4_weak})  in $C(\IR_+\,,\ctem^+)$. Then for any non-negative $\phi\in C^{\infty}_c(\IR)$, we have 
\begin{eqnarray}
\IE\left[ e^{-\la X_t, \phi \ra}\right]
&=&
 \IE\left[ e^{-\la X_0, U^{\phi}_{t}\ra - \int_{0}^{t}\la g(s,\cdot), U^{\phi}_{t-s}\ra\, ds
}\right], 
\end{eqnarray}
for all $t\geq 0.$
\end{lemma}
\begin{proof}
First, note that by standard arguments the solution $U^{\phi}$ to~\eqref{eq:mildLL} also satisfies the following weak form of the equation
\begin{eqnarray}
\label{eq:weakLL}
\la U^{\phi}_t, \psi\ra &=& \la \phi, \psi\ra + \int_0^t \la U^{\phi}_s, \Delta_{\theta}\psi\ra 
 -\frac{1}{2}U^{\phi}_s(0)^2\psi(0)\,ds,\;\;t\geq 0, \;   
\end{eqnarray}
forall $\psi\in D_{\rm tem}(\Delta_{\theta})$. Moreover $\la U^{\phi}_t, \psi\ra$ is differentialble in $t$ and
\begin{eqnarray}
\label{eq:weakLL1}
\frac{\partial \la U^{\phi}_t, \psi\ra}{\partial t} &=& \la U^{\phi}_t, \Delta_{\theta}\psi\ra 
 -\frac{1}{2}U^{\phi}_t(0)^2\psi(0),\;\;t\geq 0, \; 
\end{eqnarray}
forall $\psi\in D_{\rm tem}(\Delta_{\theta})$.

Fix arbitrary non-negative $\phi\in C^{\infty}_c(\IR)$ and $\ep>0$.  By properties of $S_t$ and spaces $\crap, \ctem$, it is easy to check that 

\begin{eqnarray}
\label{eq:ep_U} 
S_{\ep} U^{\phi}_\cdot \in C(\IR_+, D(\Delta_{\theta})), \frac{\partial S_{\ep} U^{\phi}_s}{\partial s} \in C(\IR_+,\crap),
\end{eqnarray}
and 
\begin{eqnarray}
\label{eq:ep_X}  
S_{\ep} X_t \in C(\IR_+, D_{\rm tem}(\Delta_{\theta})),\;\; {\rm a.s.}
 \end{eqnarray}
Fix arbitrary $T>0$. Use~\eqref{weakheatb}, \eqref{eq:ep_U} and~\eqref{eq:ep_X}  to get 

\begin{eqnarray}
\label{eq:DUAL1}
\la X_t, S_{\ep}U^{\phi}_{T-t} \ra
&=& \la X_0, S_{\ep} U^{\phi}_{T} \ra
\\
\nonumber
& &
+\int_{0}^{t} \lla X_s,  \Delta_{\theta} S_{\ep} U^{\phi}_{T-s} + \frac{\partial S_{\ep} U^{\phi}_{T-s}}{\partial s}\rra
 ds
\\
\nonumber
& &+\int_{0}^{t}\la g(s,\cdot), S_{\ep} U^{\phi}_{T-s}\ra\, ds
\\
\nonumber
& &+\int_{0}^{t}\sqrt{X(s,0)}S_{\ep} U^{\phi}_{T-s}(0) \,dB_s
\quad\forall t\ge 0, \quad {\rm a.s.}
\end{eqnarray}
Now use~\eqref{eq:weakLL1} and the fact that 
$$\lla U^{\phi}_{T-s}, \Delta_{\theta}p_{\ep}(x,\cdot)\rra=
\Delta_{\theta} S_{\ep} U^{\phi}_{T-s}(x),\;\; \forall x\in \IR,
$$
 to get  
\begin{eqnarray}
\label{eq:10_1_1}
\lefteqn{\Delta_{\theta} S_{\ep} U^{\phi}_{T-s}(x) + \frac{\partial S_{\ep} U^{\phi}_{T-s}}{\partial s}(x)}
\\
\nonumber
 &=&\frac{1}{2}U^{\phi}_{T-s}(0)^2 p_{\ep}(x),\;\; \forall x\in \IR.
\end{eqnarray}
Then we have 
\begin{eqnarray}
\label{eq:DUAL2}
\la X_t, S_{\ep}U^{\phi}_{T-t} \ra
&=& \la X_0, S_{\ep} U^{\phi}_{T} \ra
\\
\nonumber
& &
+\int_{0}^{t} \frac{1}{2} S_{\ep} X_s(0)U^{\phi}_{T-s}(0)^2
 ds
\\
\nonumber
& &+\int_{0}^{t}\la g(s,\cdot), S_{\ep} U^{\phi}_{T-s}\ra\, ds
\\
\nonumber
& &+\int_{0}^{t}\sqrt{X(s,0)}S_{\ep} U^{\phi}_{T-s}(0) \,dB_s
\quad\forall t\ge 0, \quad {\rm a.s.}
\end{eqnarray}

By the It\^o formula we easily get 
\begin{eqnarray}
\lefteqn{\IE\left[ e^{-\la X_t, S_{\ep}U^{\phi}_{T-t} \ra -\int_{t}^{T}\la g(s,\cdot), S_{\ep} U^{\phi}_{T-s}\ra\, ds }
\right]}
\\
\nonumber &=&
\IE\left[ e^{-\la X_0, S_{\ep}U^{\phi}_{T}\ra - \int_{0}^{T}\la g(s,\cdot), S_{\ep} U^{\phi}_{T-s}\ra\, ds
}\right]
\\
\nonumber &&\mbox{}+\IE\left[\int_0^t 
e^{-\la X_0, S_{\ep}U^{\phi}_{T-s}\ra - \int_{s}^{T}\la g(r,\cdot), S_{\ep} U^{\phi}_{T-r}\ra \,dr} 
\right.
\\
\nonumber
&& \;\;\;\;\left. \mbox{}\times \frac{1}{2}  
\left\{ X(s,0)(S_{\ep} U^{\phi}_{T-s}(0))^2- S_{\ep} X_s(0)U^{\phi}_{T-s}(0)^2   \right\}\,ds\right].
\end{eqnarray}
Now let $\ep\rightarrow 0$. Use the continuity of $X_t(\cdot)$, and $U^{\phi}_{t}(\cdot)$ and the dominated convergence 
theorem to get 
\begin{eqnarray}
\lefteqn{\IE\left[ e^{-\la X_t, U^{\phi}_{T-t} \ra -\int_{t}^{T}\la g(s,\cdot),  U^{\phi}_{T-s}\ra\, ds}\right]}
\\
\nonumber&=&
 \IE\left[ e^{-\la X_0, U^{\phi}_{T}\ra - \int_{0}^{T}\la g(s,\cdot), U^{\phi}_{T-s}\ra\, ds
}\right],
\end{eqnarray}
for all $t\in [0,T]$. By taking $t=T$, we get 
\begin{eqnarray}
\IE\left[ e^{-\la X_T, \phi \ra}\right]
&=&
 \IE\left[ e^{-\la X_0, U^{\phi}_{T}\ra - \int_{0}^{T}\la g(s,\cdot), U^{\phi}_{T-s}\ra\, ds
}\right], \;\;\forall T>0,
\end{eqnarray}
and we are done. 
\end{proof}

\paragraph{Proof of Theorem~\ref{thm:unique_weak_SPDE}}
By Lemma~\ref{lem:uniq_onedim} we immediately get that for any solution $X_{\cdot}$ of 
(\ref{equt:4_weak})  in $C(\IR_+\,,\ctem^+)$ the law of 
 $X_t\in \ctem^+$ is unique for any $t>0$, or in other words the uniqueness of one-dimensional distributions holds. 
 This is true for any initial conditions  $X_0\in \ctem^+$.
By standard argument this implies also uniqueness of one dimensional distribuitions (see e.g. Theorem 4.4.2 and its 
 proof in~\cite{bib:kur86}), and hence weak uniqueness for solutions of~(\ref{equt:4_weak})  in $C(\IR_+\,,\ctem^+)$.
 
\hfill$\square$

\medskip

\paragraph{Proof of Theorem~\ref{thm:weakUniq}}
This follows immediately by corresspondence of solutions to (\ref{equt:4_weak}) and \eqref{equt:cat_pos}
(see Lemma~\ref{thm:nonLipschitzexistence} and discussion after it). 

\hfill$\square$

\medskip

Note that, in fact, $U^{\phi}$ is the so-called log-Laplace equation for the catalytic superprocess with single point catalyst at $0$, and the 
motion process generated by $\Delta_{\theta}$. So in principal, we could prove
Theorem~\ref{thm:unique_weak_SPDE}
by showing that such 
any such superprocess is in fact a weak solution to (\ref{equt:4_weak}).
We gave the more detailed  proof just for the sake of 
completeness.

\section{Uniqueness: preliminary estimates}

\label{section:auxiliary}
In this section we will develop machinery for proving Theorem~\ref{thm:uniquekunbounded}.
The proof follows a similar approach to that in~\cite{bib:msp05}.

Let $\rho$ be a strictly
increasing function on
$\IR_{+}$ such that
\begin{equation}\label{rhobnd}
\rho(x)\ge \sqrt x.
\end{equation}
and
\begin{equation}
\label{rhocond}
\int_{0+} \rho^{-2}(x) dx = \infty.
\end{equation}
As in the proof of Yamada and Watanabe
\cite{bib:YW71}, we may define  a sequence of functions $\phi_{n}$ in the
following way.  First, let $a_{n} \downarrow 0$ be a strictly
decreasing sequence such that $a_{0}=1$, and
\begin{equation}
\label{acond}
\int_{a_{n}}^{a_{n-1}} \rho^{-2}(x) dx = n.
\end{equation}
Second, we define functions $\psi_{n} \in C^{\infty}_{c}(\IR)$ such that
$supp(\psi_{n}) \subset  (a_{n}, a_{n-1})$, and that
\begin{equation}
\label{psicond}
0 \leq \psi_{n}(x) \leq \frac{2 \rho^{-2}(x)}{n}\leq {2\over nx}
\quad \mbox{ for all $x \in \IR$ as well as } \quad \int_{a_{n}}^{a_{n-1}} \psi_{n}(x) dx =1.
\end{equation}
Finally, set
\begin{equation}
\label{def:phi}
\phi_{n}(x) = \int_{0}^{|x|} \int_{0}^{y} \psi_{n}(z) dz dy.
\end{equation}
From this it is easy to see that $\phi_{n}(x) \uparrow |x|$ uniformly in $x\geq 0.$
Note that each $\psi_{n}$ and thus also each $\phi_{n}$ is
identically zero in a neighborhood of zero. This implies that
$\phi_{n} \in C^{\infty}(\IR)$ despite the absolute value in its
definition. We have
\begin{eqnarray*}
\label{phidiff1}
\phi_{n}'(x) &=& \sgn(x) \int_{0}^{|x|}  \psi_{n}(y) dy,\\
\label{phidiff2}
\phi_{n}''(x) &=&  \psi_{n}(|x|).
\end{eqnarray*}
Thus, $|\phi_{n}'(x)| \leq 1$, and
$\int \phi_{n}''(x) h(x) dx \rightarrow h(0)$ for any function $h$ which
is continuous at zero.

\smallskip
Now let $X^{1}$ and $X^{2}$ be two solutions of (\ref{equt:4}) with sample
paths in $C(\IR_+,\ctem)$ a.s.,
 with the same
initial condition, $X^{1}(0)=X^{2}(0)=X_0\in \ctem$, and the same 
Brownian motion $B$
in the setting of Theorem
\ref{thm:uniquekunbounded}.  
  Define $\tilde{X} \equiv X^{1} - X^{2}.$ 
 Set
$\Phi_{x}^{m}( y) = p^{\theta}_{m^{-1/\alpha}}(x,y)$.
\noindent
Note that for any $x\in\IR$,  $\Phi_{x}^{m}( \cdot)\in D(\Delta_{\theta})$.  
 Use~\eqref{weakheatb} to get the the semimartingale decomposition of 
$\la \tilde{X}_{t}, \Phi^{m}_{x} \ra=\la  X^{1}_t - X^{2}_t, \Phi^{m}_{x} \ra$. Then apply 
It\^{o}'s Formula to the semimartingale
$\la \tilde{X}_{t}, \Phi^{m}_{x} \ra$ 
to get
\begin{align*}
& \phi_{n}(\langle \tilde{X}_{t}, \Phi_{x}^{m} \ra)\\
&\qquad= \int_{0}^{t} \int_{\IR} \phi_{n}'(\langle \tilde{X}_{s}, \Phi_{x}^{m} \ra)
\left( \sigma(X^{1}(s,0)) - \sigma(X^{2}(s,0)) \right)
\Phi_{x}^{m}(0)  \,dB_s\\
& \qquad\quad+\int_{0}^{t} \phi_{n}'(\langle \tilde{X}_{s}, \Phi_{x}^{m} \ra)
\langle \tilde{X}_{s},  \Delta_{\theta} \Phi_{x}^{m} \ra ds\\
& \qquad\quad+ \frac{1}{2}
\int_{0}^{t} 
\psi_{n}(|\langle \tilde{X}_{s}, \Phi_{x}^{m} \ra|)
\left( \sigma(X^{1}(s,0)) - \sigma(X^{2}(s,0)) \right)^2
\Phi_{x}^{m}(0)^2 ds.
\end{align*}
We integrate this function of $x$ against another non-negative test
function
$\Psi \in C([0,t], D(\Delta_{\theta}))$ such that 
\begin{eqnarray}
\label{equt:33}
\Psi_s(0)>0,\;\forall s\geq 0\;\;{\rm and}\;\;
\sup_{s\leq t} \left|\int_{\IR} |x|^{-\theta}\left(\frac{\partial \Psi_s(x)}{\partial x}\right)^2\,dx\right|<\infty,\;\;\forall t>0,
\end{eqnarray}
and $s\mapsto\frac{\partial \Psi_s(\cdot)}{\partial s}\in C(\IR_+,\crap).$
Also assume
$\Gamma(t)\equiv\{x:\exists s\le t,\,\Psi_s(x)>0\}
\subset B(0,J(t))$  for  some $J(t)>0.$ We then obtain by the
classical and stochastic version of Fubini's Theorem, and arguing as
in the proof of Proposition II.5.7 of \cite{bib:perk99} to handle the time
dependence in $\Psi$, that for any $t\geq 0$,
\begin{align}
\label{eq:Iparts}
& \left\langle \phi_{n}(\langle \tilde{X}_{t}, \Phi_{.}^{m} \ra), \Psi_t
\right\ra\\
\nonumber
&\qquad= \int_{0}^{t}
\langle \phi_{n}'(\langle \tilde{X}_{s}, \Phi_{\cdot}^{m} \ra) \Phi_{\cdot}^{m}(0)
 , \Psi_s \ra \left( \sigma(X^{1}(s,0)) - \sigma(X^{2}(s,0)) \right)\,
dB_s\\
\nonumber
& \qquad\qquad+ \int_{0}^{t} \langle \phi_{n}'(\langle \tilde{X}_{s}, \Phi_{.}^{m} \ra)
\langle \tilde{X}_{s},  \Delta_{\theta} \Phi_{.}^{m} \ra, \Psi_s \ra ds\\
\nonumber
& \qquad\qquad+ \frac{1}{2}
\int_{0}^{t} \langle
\psi_{n}(|\langle \tilde{X}_{s}, \Phi_{\cdot}^{m} \ra|)\Phi^m_\cdot(0)^2,\Psi_s\rangle
\left( \sigma(X^{1}(s,0)) - \sigma(X^{2}(s,0)) \right)^2\,ds
\\
\nonumber
&\qquad\qquad+\int_0^t\langle\phi_n(\langle\tilde
X_s,\Phi^m_\cdot\ra) ,\dot\Psi_s\ra\,ds\\
\nonumber
&\qquad \equiv I_{1}^{m,n}(t) + I_{2}^{m,n}(t) + I_{3}^{m,n}(t)+I_{4}^{m,n}(t).
\end{align}

\noindent

We need a calculus lemma.  For $f\in C^2(\IR)$, let $\Vert
D^2f\Vert_\infty=\Vert{\partial^2 f\over \partial
x^2}\Vert_\infty$.

\begin{lemma}\label{calc} Let $f\in C_c^2(\IR)$ be non-negative and not
identically zero.  Then
$$\sup\Bigl\{\Bigl({{\partial f\over \partial x}(x)}\Bigr)^2
f(x)^{-1}:f(x)>0\Bigr\}\le 2\Vert D^2f\Vert_\infty.$$
\end{lemma}

\begin{proof} See Lemma~2.1 of~\cite{bib:msp05}.

\end{proof}

We now consider the expectation of
expression (\ref{eq:Iparts}) stopped at a stopping time $T,$
that we will choose later on. Ultimately we will use the following to show that the contributions of $I^{m,n}_1$, $I^{m,n}_2$, and $I^{m,n}_4$ to this expectation disappear in the limit. $I^{m,n}_3$ is where the classical Yamada-Watanabe calculation
comes into play, to be analyzed in Section~\ref{section:uniquenessargument}
\begin{lemma}\label{lem:conv} Assume the hypotheses of Theorem~\ref{thm:uniquekunbounded}. For any stopping time $T$ and constant $t\ge0$ we
have:
\begin{enumerate}
\item $\displaystyle\IE(I_1^{m,n}(t\wedge T))=0\hbox{ for all }m,n.$
\item 
$\displaystyle\limsup_{m,n\rightarrow \infty} \IE( I_{2}^{m,n}(t \wedge T))
\le  \IE \Big( \int_{0}^{t \wedge T} \int_{\IR}
|\tilde{X}(s,x)|  \Delta_{\theta}\Psi_s(x)dx ds \Big).$
\item 
$\displaystyle\lim_{m,n\rightarrow \infty} \IE(I_4^{m,n}(t\wedge
T))=\IE\Bigl(\int_0^{t\wedge T}|\tilde X(s,x)|\dot\Psi_s(x)\,ds\Bigr).$
\end{enumerate}
\end{lemma}

\begin{proof}
(a) Let $g_{m,n}(s)=\langle\phi_n'(\langle \tilde
X_s,\Phi^m_\cdot\rangle)\Phi^m_\cdot(0),\Psi_s\rangle$.  Note first that
$I_1^{m,n}(t\wedge T)$ is a continuous local martingale with square function
\begin{eqnarray}
\nonumber
\langle I_1^{m,n}\rangle_{t\wedge T}&=&\int_0^{t\wedge
T}g_{m,n}(s)^2(\sigma(X^1(s,0))-\sigma(X^2(s,0)))^2
ds\\
\nonumber &\le& C\int_0^{t\wedge T}
 g_{m,n}(s)^2(|X^1(s,0)|+|X^2(s,0)|+2)^2\,ds.
\end{eqnarray}
An easy calculation shows that $|g_{m,n}(s,y)|\le \Vert \Psi\Vert_\infty$, so by (\ref{eq:uniformmoments2a})
\[\IE(\langle I_1^{m,n}\rangle_{t\wedge T})\le C(t)<\infty\ \forall t>0.\]
This shows $I_1^{m,n}(t\wedge T)$ is a square integrable martingale and so has
mean $0$, as required.  

(b) We have  to rewrite $I_{2}^{m,n}$. 
Denote by $\Delta_{x,\theta}$ the $\theta$-Laplacian acting with respect to $x.$
We know by symmetry that 
\[ \Delta_{y,\theta} \Phi_{x}^{m}(y)=\Delta_{x,\theta} \Phi_{x}^{m}(y).\]
Hence, since $\tilde{X}_{s}$ is locally integrable and continuous
 we have for $|x|\le J(t)$,
\[
\label{Deltaex}
\int_{\IR} \tilde{X}(s,y)   \Delta_{y,\theta}\Phi^{m}_x(y) dy
= \int_{\IR} \tilde{X}(s,y)   \Delta_{x,\theta}\Phi^{m}_x(y) dy
=  \Delta_{x,\theta} \int_{\IR} \tilde{X}(s,y) \Phi^{m}_x(y) dy,
\]
for all $m.$
This implies for any $t\geq 0$,
\begin{eqnarray*}
\nonumber
I_{2}^{m,n}(t)&=&
\int_{0}^{t}\int_{\IR}\phi_{n}'(\la \tilde{X}_{s}, \Phi_{x}^{m} \rangle)
   \Delta_{x,\theta}\left(\la \tilde{X}_{s},
\Phi_{x}^{m} \rangle\right) \Psi_s(x) dx ds\\
\nonumber &=&
- 2\alpha^2
\int_{0}^{t}\int_{\IR} \frac{\partial}{\partial x}
\left(\phi_{n}'(\la \tilde{X}_{s}, \Phi_{x}^{m} \rangle)\right)
|x|^{-\theta}\frac{\partial}{\partial x}
\left( \la \tilde{X}_{s}, \Phi_{x}^{m} \rangle \right) \Psi_s(x) dx ds\\
\nonumber & &- 2\alpha^2
\int_{0}^{t}\int_{\IR} \phi_{n}'(\la \tilde{X}_{s}, \Phi_{x}^{m} \rangle)
|x|^{-\theta}\frac{\partial}{\partial x}
\left( \la \tilde{X}_{s}, \Phi_{x}^{m} \rangle \right)
\frac{\partial}{\partial x}\Psi_s(x) dx ds\\
\nonumber &=&-2\alpha^2  \int_{0}^{t}\int_{\IR}
\psi_{n}(|\la \tilde{X}_{s}, \Phi_{x}^{m} \rangle|)
|x|^{-\theta}\left( \frac{\partial}{\partial x}
\la \tilde{X}_{s}, \Phi_{x}^{m} \rangle \right)^{2} \Psi_s(x) dx ds\\
\nonumber & & -2\alpha^2
\int_{0}^{t}\int_{\IR}\phi_{n}'(\la \tilde{X}_{s}, \Phi_{x}^{m} \rangle)
|x|^{-\theta}\frac{\partial}{\partial x}
\left( \la \tilde{X}_{s}, \Phi_{x}^{m} \rangle  \right)
\frac{\partial}{\partial x}\Psi_s(x) dx ds\\
\nonumber
&=& -  2\alpha^2   \int_{0}^{t}\int_{\IR}
\psi_{n}(|\la \tilde{X}_{s}, \Phi_{x}^{m} \rangle|)
|x|^{-\theta}\left( \frac{\partial}{\partial x}
\la \tilde{X}_{s}, \Phi_{x}^{m} \rangle \right)^{2} \Psi_s(x) dx ds\\
\nonumber
& &+2\alpha^2\int_{0}^{t}\int_{\IR}
\psi_{n}(\la \tilde{X}_{s},\Phi_{x}^{m}\rangle)|x|^{-\theta}\frac{\partial}{\partial x}
(\la \tilde{X}_{s},\Phi_{x}^{m}\rangle)
\la \tilde{X}_{s},\Phi_{x}^{m}\rangle \frac{\partial}{\partial x} \Psi_s(x)
dx ds\\
\nonumber
& &+ \int_{0}^{t}\int_{\IR}
\phi_{n}'(\la \tilde{X}_{s},\Phi_{x}^{m}\rangle)
\la \tilde{X}_{s},\Phi_{x}^{m}\rangle
 \Delta_{\theta} \Psi_s(x)dx ds\\
\label{secondterm}
&=& \int_{0}^{t} I_{2,1}^{m,n}(s) + I_{2,2}^{m,n}(s) + I_{2,3}^{m,n}(s) ds.
\end{eqnarray*}
Above, we have used that $\phi_{n}''=\psi_{n}$ and we have repeatedly
used integration by parts, the product rule as well as the chain rule on
$\phi_{n}'(\la \tilde{X}_{s}, \Phi_{x}^{m} \rangle).$
In order to deal with the various parts of $I_{2}^{m,n}$
we will first jointly consider $I_{2,1}^{m,n}$ and $I_{2,2}^{m,n}.$
For fixed $s$ we define a.s.,
\begin{eqnarray*}
A^{s}&=& \Bigl\{ x: \left( \frac{\partial}{\partial x}
\la \tilde{X}_{s}, \Phi_{x}^{m} \rangle \right)^{2} \Psi_s(x) \leq
\la \tilde{X}_{s}, \Phi_{x}^{m} \rangle \frac{\partial}{\partial x}
\la \tilde{X}_{s}, \Phi_{x}^{m} \rangle \frac{\partial}{\partial x}
\Psi_s(x) \Bigr\} \cap \{x:\Psi_s(x)>0\}\\
&=& A^{+,s}\cup A^{-,s} \cup A^{0,s},
\end{eqnarray*}
where
\begin{eqnarray*}
A^{+,s}&=& A^{s} \cap \{ \frac{\partial}{\partial x}
\la \tilde{X}_{s}, \Phi_{x}^{m} \rangle >0\},\\
A^{-,s} &=& A^{s} \cap \{ \frac{\partial}{\partial x}
\la \tilde{X}_{s}, \Phi_{x}^{m} \rangle <0\},\\
A^{0 ,s}&=& A^{s}  \cap \{ \frac{\partial}{\partial x}
\la \tilde{X}_{s}, \Phi_{x}^{m} \rangle =0\}.
\end{eqnarray*}
By~(\ref{equt:33}) we can find $\ep>0$ sufficiently small such that 
\begin{eqnarray}
\label{equt:34}
 B(0,\ep)\subset \Gamma(t),\;\;{\rm and}\;\; \inf_{s\leq t, x\in B(0,\ep)} \Psi_s(x)>0.
\end{eqnarray}
On $A^{+,s}$ we have
\begin{equation*}
0 <\Bigl(\frac{\partial}{\partial x}
\la \tilde{X}_{s}, \Phi_{x}^{m} \rangle \Bigr)\Psi_s(x)\leq
\la \tilde{X}_{s}, \Phi_{x}^{m} \rangle \frac{\partial}{\partial x}
\Psi_s(x),
\end{equation*}
and therefore for any $t\geq 0$,
\begin{align*}
&\int_0^t\int_{A^{+,s}} \psi_{n}(|\la \tilde{X}_{s}, \Phi_{x}^{m} \rangle
|)|x|^{-\theta}
\la \tilde{X}_{s}, \Phi_{x}^{m} \rangle \frac{\partial}{\partial x} \Psi_s(x)
\frac{\partial}{\partial x}
\la \tilde{X}_{s}, \Phi_{x}^{m} \rangle dx\,ds\\
&\qquad\leq \int_0^t\int_{A^{+,s}} \psi_{n}(|\la \tilde{X}_{s}, \Phi_{x}^{m}
\rangle |)|x|^{-\theta}
\la \tilde{X}_{s}, \Phi_{x}^{m} \rangle^{2} \frac{(\frac{\partial}
{\partial x} \Psi_s(x))^{2}}{\Psi_s(x)} dx\,ds\\
&\qquad\leq \int_0^t\int_{A^{+,s}} \frac{2}{n}
1_{\{a_{n-1} \leq |\la \tilde{X}_{s}, \Phi_{x}^{m} \rangle| \leq a_{n} \}}
|x|^{-\theta}|\la \tilde{X}_{s}, \Phi_{x}^{m} \rangle| \frac{(\frac{\partial}
{\partial x} \Psi_s(x))^{2}}{\Psi_s(x)} dx\,ds\quad\hbox{by (\ref{psicond})}
\\ &\qquad\leq \frac{2a_{n}}{n} \int_0^t\int_{\IR} 1_{\{\Psi_s(x)>0\}}|x|^{-\theta}
\frac{(\frac{\partial}{\partial x} \Psi_s(x))^{2}}{\Psi_s(x)} dx\,ds\\
&\qquad\le \frac{2a_n}{n}\int_0^t 
 \left(\int_{B(0,\ep)}
 \frac{(\frac{\partial}{\partial x} \Psi_s(x))^{2}}{\Psi_s(x)}  |x|^{-\theta} \,dx
+
2\Vert
D^2\Psi_s\Vert_\infty\int_{\Gamma\smallsetminus B(0,\ep)} |x|^{-\theta} \,dx\,\right)ds\\
&\qquad\equiv {2a_n\over n}C(\Psi,t),
\end{align*}
where (\ref{equt:33}), (\ref{equt:34})   and  Lemma~\ref{calc} are used in the last two lines.
Similarly, on the set $A^{-,s}_{i},$
\begin{equation*}
0 > \Big(\frac{\partial}{\partial x}
\la \tilde{X}_{s}, \Phi_{x}^{m} \rangle \Big)\Psi_s(x) \geq
\la \tilde{X}_{s}, \Phi_{x}^{m} \rangle \frac{\partial}{\partial x}
\Psi_s(x).
\end{equation*}
Hence, with the same calculation
\begin{align*}
&\int_0^t\int_{A^{-,s}} \psi_{n}(|\la \tilde{X}_{s}, \Phi_{x}^{m} \rangle
|)|x|^{-\theta}
\la \tilde{X}_{s}, \Phi_{x}^{m} \rangle \frac{\partial}{\partial x} \Psi_s(x)
\frac{\partial}{\partial x}
\la \tilde{X}_{s}, \Phi_{x}^{m} \rangle dx\,ds\\
&\qquad\leq \frac{2a_{n}}{n}\int_0^t \int_{\IR} 1_{\{\Psi_s(x)>0\}}|x|^{-\theta}
\frac{(\frac{\partial}{\partial x} \Psi_s(x))^{2}}{\Psi_s(x)} dx\,ds\\
&\qquad\leq {2a_n\over n}C(\Psi,t).
\end{align*}
Finally, for any $t\geq 0$,
\begin{equation*}
\int_0^t\int_{A^{0,s}} \psi_{n}(|\la \tilde{X}_{s}, \Phi_{x}^{m} \rangle |)
|x|^{-\theta}
\la \tilde{X}_{s}, \Phi_{x}^{m} \rangle \frac{\partial}{\partial x} \Psi_s(x)
\frac{\partial}{\partial x}
\la \tilde{X}_{s}, \Phi_{x}^{m} \rangle dx\,ds=0,
\end{equation*}
and we conclude that
\begin{equation*}
\IE( I_{2,1}^{m,n}(t \wedge T)
+ I_{2,2}^{m,n}(t \wedge T)) \leq 4\alpha^2C(\Psi,t)
 \frac{a_{n}}{n},
\end{equation*}
which tends to zero as $n \rightarrow \infty.$ 

For $I_{2,3}^{m,n}$
recall that $\phi_{n}'(X) X \uparrow |X|$ uniformly in $X$ as
$n \rightarrow \infty,$
and that $\la \tilde{X}_{s}, \Phi_{x}^{m} \rangle$
tends to $\tilde{X}(s,x)$ as $m \rightarrow
\infty$ for all $s,x$ a.s. by the a.s. continuity of $\tilde X$.
This implies
that
$\phi_{n}'(\la \tilde{X}_{s}, \Phi_{x}^{m} \rangle)
\la \tilde{X}_{s}, \Phi_{x}^{m} \rangle
\rightarrow |\tilde{X}(s,x)|$
pointwise a.s.
as $m,n \rightarrow \infty,$ where it is unimportant how we take the limit.
We also have the bound
\begin{eqnarray}
\label{phinbound}
|\phi_{n}'(\la \tilde{X}_{s}, \Phi_{x}^{m} \rangle)
\la \tilde{X}_{s}, \Phi_{x}^{m} \rangle|
&\leq&
|\la \tilde{X}_{s}, \Phi_{x}^{m} \rangle|
\leq \la |\tilde{X}_{s}|, \Phi_{x}^{m} \rangle.
\end{eqnarray}
The a.s. continuity of $\tilde X$ implies a.s. convergence for all $s,x$ of
$\langle|\tilde X_s|,\Phi^m_x\rangle$ to $|\tilde X(s,x)|$ as $m\to \infty$. 
Jensen's Inequality and (\ref{eq:uniformmoments2a}) show
that
$\langle|\tilde X_s|,\Phi^m_x\rangle \,$ is $L^p$
bounded on $([0,t]\times B(0,J(t))\times \Omega,ds\times
dx\times \IP)$ uniformly in $m.$  Therefore
\begin{equation}\label{ui}
\{\la |\tilde{X}_{s}|, \Phi_{x}^{m} \ra
:m\}\hbox{ is uniformly integrable
on }([0,t]\times B(0,J(t))\times \Omega).
\end{equation}
This gives uniform
integrability of $
\{|\phi_{n}'(\la \tilde{X}_{s}, \Phi_{x}^{m} \rangle)
\la \tilde{X}_{s}, \Phi_{x}^{m} \rangle|:m,n\}$ by our earlier bound
(\ref{phinbound}).
Since $\Psi_s=0$ off $B(0,J(t))$, this implies that
 $$\lim_{m,n\to\infty}\IE( I_{2,3}^{m,n}(t\wedge
T))=\IE\Bigl(\int_0^{t\wedge T}\int|\tilde X(s,x)|\Delta_{\theta} \Psi_s(x)dx\,ds\Bigr).$$  Collecting the pieces, we have shown that
(b) holds.

\medskip

\noindent(c) As in the above argument we have
\begin{equation}\label{absconv}
\phi_n(\la\tilde X_s,\Phi_x^m\rangle)\to|\tilde
X(s,x)|\hbox{ as }m,n\to\infty\hbox{ a.s. for all }x\hbox{ and all }s\le t.
\end{equation}
The uniform integrability in (\ref{ui})
and the bound $\phi_n(\la\tilde X_s,\Phi^m_x\rangle)\le\la|\tilde
X_s|,\Phi_x^m\rangle$ imply that
$$\{\phi_n(\la\tilde X_s,\Phi_x^m\rangle:n,m\}\hbox{ is uniformly integrable
on }[0,t]\times B(0,J(t))\times\Omega.$$
Therefore the result now follows from the above convergence and the
bound $$|\dot\Psi_s(x)|\le
C1_{\{|x|\le J(t)\}}.$$ \end{proof}

\section{Uniqueness: Theorem~\ref{thm:uniquekunbounded}}
\label{section:uniquenessargument}
Let $T_K=\inf\{t\ge 0: \sup_{x \in [-1,1]} (|X^1(t,x)|+|X^2(t,x)|)>K\}
\wedge K.$ 
Note that 
\begin{eqnarray}
\label{equt:120}
T_{K}\rightarrow \infty, \text{ $\IP$-a.s. as $K\to\infty$}\;
\end{eqnarray}
since each $X^i$ is continuous.
Also define a metric
$d$ by
\[d((t,x),(t',x'))= |t-t'|^{\alpha}+|x-x'|, t,t'\in\IR_+, x,x'\in\IR,\]
and set
\begin{eqnarray*}
 \lefteqn{Z_{K,N,\xi}\equiv\left\{(t,x)\in \Re_+\times\Re: t\leq T_K, \; |x|\leq 2^{-N\alpha},\; 
 \lmed t-\hat{t}\rmed\leq 2^{-N}, |x-\hat{x}|\leq 2^{-N\alpha},\;\right. }
\\
&&\left.
{\rm for\; some}\; 
(\hat{t},\hat{x})\in    [0,T_K]\times\Re\; {\rm satisfying}\; \lmed\wX(\hat{t},\hat{x})\rmed \leq 2^{-N\xi}\right\}
\end{eqnarray*}

We will now use the following key result on  improving the  H\"older continuity of
$\tilde X(t,x)$ when $\tilde X$ and $|x|$ are  small. We will assume this result in this section, where we will use it to show uniqueness.
We will prove Theorem~\ref{thm:H1} in Section~\ref{sec:proofofimprovement}.
\begin{theorem}
\label{thm:H1}
Assume the hypotheses of Theorem~\ref{thm:uniquekunbounded}, 
$\wX=X^1-X^2$,
where $X^i$ is a solution of (\ref{equt:4}) with sample paths in
$C(\IR_+,\ctem)$ a.s. for
$i=1,2$.
Let $\xi\in (0,1)$ satisfy
\begin{eqnarray}
\nonumber
\exists N_{\xi} &=&N_{\xi}(K,\omega)\in\IN\
\text{a.s. such that for any $N\ge N_{\xi}$, and any }
(t,x)\in Z_{K,N,\xi}
\end{eqnarray}
\begin{eqnarray}
\label{reghyp}
\left. \begin{array}{r}
|t'-t|\le 2^{-N}, t,t'\le
T_K
\\
|y-x|\le 2^{-N\alpha}        
       \end{array}
\right\}
&\Rightarrow& |\wX(t,x)-\wX(t',y)|\le
2^{-N\xi}.
\end{eqnarray}
Let $\frac{1}{2}-\alpha<\xi^1<[\xi \gamma + \frac{1}{2}-\alpha]\wedge 1$.
Then there is an $N_{\xi^1}=N_{\xi^1}(K,\omega,\xi)\in\IN$
a.s. such that for any $N\ge N_{\xi^1}$ in $\IN$ and any $(t,x)\in Z_{K,N,\xi^1}$
\begin{eqnarray}
\label{reghyp2}
\left. \begin{array}{r}
|t'-t|\le 2^{-N}, t,t'\le
T_K
\\
|y-x|\le 2^{-N\alpha}        
       \end{array}
\right\}
&\Rightarrow& |\wX(t,x)-\wX(t',y)|\le
2^{-N\xi^1}.
\end{eqnarray}
Moreover there are strictly positive constants
$R,\delta, c_{\ref{Nbnd}.1},c_{\ref{Nbnd}.2}$ depending only on
$(\xi,\xi^1)$, and $N(K) \in \IN$ which also depends on $K,$ such
that
\begin{eqnarray}
\label{Nbnd}
\IP(N_{\xi^1}\ge N)\le c_{\ref{Nbnd}.1} ( \IP(N_{\xi}\ge N/R)
+ K \exp(-c_{\ref{Nbnd}.2} 2^{N\delta}) )
\end{eqnarray}
provided that $N\geq N(K).$
\end{theorem}

\begin{corollary}
\label{cor:ubnd}
Assume the hypthoses of Theorem~\ref{thm:uniquekunbounded}.
  Let $\tilde X$ be as in
Theorem~\ref{thm:H1}, and $\frac12-\alpha<\xi<
\frac{\frac12-\alpha}{1-\gamma} \wedge 1$. There is an a.s. finite positive
random variable
$C_{\xi,K}(\omega)$ such that for any $\epsilon\in(0,1]$, $t\in[0,T_K]$ and $|x|\le
\ep^{\alpha}$, if $|\tilde X(t,\hat{x})|\le \epsilon^\xi$ for some $|\hat{x}-x|\le \epsilon^{\alpha}$, then
$|\tilde X(t,y)| \le C_{\xi,K}\epsilon^\xi$ whenever $|x-y|\le \epsilon^{\alpha}$.  Moreover
there are strictly positive constants
$\delta,c_{\ref{Cprobbnd}.1},c_{\ref{Cprobbnd}.2}$, depending on $\xi$, and an
$r_0(K),$ which also depends on $K,$ such that
\begin{eqnarray}
\label{Cprobbnd}
\IP(C_{\xi,K}\ge r)\le c_{\ref{Cprobbnd}.1}\Bigl[ \Bigl(\frac{r-6}{K+1}
\Bigr)^{-\delta} + K
\exp\Bigl(-c_{\ref{Cprobbnd}.2}\Bigl({r-6\over K+1}\Bigr)^\delta\Bigr)\Bigr]
\end{eqnarray}
for all $r\geq r_0(K)>6+(K+1).$
\end{corollary}
\begin{proof}
Let $\xi_0=\frac{\alpha}{2}(\frac12-\alpha)$. By  Proposition~\ref{prop}(b) and the equality $\tilde X=Z^1-Z^2$, where
$Z^i(t,x)=X^i(t,x)-S_t X_0(x)-\int_0^t \int_{\Resm}p^{\theta}_{t-s}(x,y) g(s,y)\,dy\,ds$, we have (\ref{reghyp})  with
$\xi=\xi_0$.  Indeed, $\tilde{X}$ is
uniformly H\"older continuous on compacts in $[0,\infty)\times\IR$
in space and in time  with any exponent less than $\frac12-\alpha$. This allows to get \eqref{reghyp} with $\xi=\xi_0$.

Inductively define $\xi_{n+1}=\Bigl[\Bigl( \xi_{n} \gamma + \frac12
-\alpha\Bigr)
\wedge 1\Bigr]\Bigl(1-{1\over n+3}\Bigr)$. It is easily checked that $\xi_0<\xi_1$, from which it follows inductively 
that
$\xi_n\uparrow \frac{\frac12-\alpha}{1-\gamma}\wedge 1$. Let now $\xi$ be as in the statement of the corollary, that is 
$\xi\in (\frac12-\alpha,
\frac{\frac12-\alpha}{1-\gamma} \wedge 1)$. 
Fix $n_0$ so that $\xi_{n_0}\ge \xi>\xi_{n_0-1}$.  Apply
Theorem~\ref{thm:H1} inductively $n_0$ times
to get (\ref{reghyp}) for
$\xi=\xi_{n_0-1}$ and, hence, 
(\ref{reghyp2}) with $\xi^1=\xi_{n_0}$.  

First consider $\epsilon\le 2^{-N_{\xi_{n_0}}}$.  Choose $N\in\IN$ so that
$2^{-N-1}<\epsilon\le 2^{-N}$, and so  $N\ge N_{\xi_{n_0}}$. Assume $t\le T_K$, $|x|\le
\ep^{\alpha}\le2^{-N\alpha}$, and
$|\tilde X(t,\hat{x})|\le \epsilon^\xi\le 2^{-N\xi}\le 2^{-N\xi_{n_0-1}}$ for some
$|\hat{x}-x|\le \epsilon^{\alpha}\le 2^{-N\alpha}$. Then $(t,x)\in Z_{K,N,\xi_{n_0-1}}$.  Therefore
(\ref{reghyp2}) with $\xi^1=\xi_{n_0}$ implies that if $|y-x|\le \epsilon^{\alpha}\le
2^{-N\alpha}$, then
\begin{eqnarray*}
|\tilde X(t,y)|&\le&|\tilde X(t,\hat{x})|+|\tilde X(t,\hat{x})-\tilde X(t,x)|
+|\tilde X(t,x)-\tilde X(t,y)|\\
&\le &2^{-N\xi}+2\cdot 2^{-N\xi_{n_0}}
\le 3\cdot 2^{-N\xi}\le 3(2\epsilon)^\xi\le 6\epsilon^\xi.
\end{eqnarray*}
For $\epsilon>2^{-N_{\xi_{n_0}}}$, we have for $(t,x)$
and $(t,y)$ as in the corollary,
\[|\tilde X(t,y)|\le K+1
\le (K+1) 2^{N_{\xi_{n_0}}\xi}\epsilon^\xi. \]
This gives the conclusion with $C_{\xi,K}=(K+1)  2^{N_{\xi_{n_0}}\xi}+6$.
A short calculation and (\ref{Nbnd}) now imply that there are strictly positive
constants
$\tilde{R},\tilde{\delta},{c}_{\ref{Cprobbnd1}.1},{c}_{\ref{Cprobbnd1}.2}$,
depending on $\xi$ and $K$, such that 
\begin{eqnarray}
\label{Cprobbnd1}
\IP(C_{\xi,K}\ge r)&\le&
{c}_{\ref{Cprobbnd1}.1}\Bigl[\IP\Bigl(N_{\frac{1}{2}(\frac12-\alpha) }\ge
\frac{1}{\tilde{R}} \log_2 \Bigl({r-6\over K+1}\Bigr)
\Bigr)\\
&&\phantom{\tilde{c}_{\ref{Cprobbnd1}.1}\Bigl[\IP\Bigl(N_{\frac{\alpha}{2}(1-\frac{\alpha}{2})
}\ge \frac{1}{R} \log_2}
\nonumber
+ K\exp\Bigl(-{c}_{\ref{Cprobbnd1}.2}\Bigl({r-6\over K+1}\Bigr)^{\tilde{\delta}}\Bigr)\Bigr]
\end{eqnarray}
for all $r\geq r_0(K).
$
The usual Kolmogorov continuity proof applied to 
(\ref{eq:Kolmogorov}) with
$\tilde{X}=Z^1-Z^2$ in place of $Z$ (and $\xi=\frac{\alpha}{2}(\frac12-\alpha)$)
shows there are $\tilde{\epsilon},\tilde{c}_{3} >0$
such that
$$\IP(N_{\frac{\alpha}{2}(\frac12-\alpha)}\geq M)
\leq \tilde{c}_{3} 2^{-M \tilde{\epsilon}}$$
for all $M \in \IR.$
Thus, (\ref{Cprobbnd}) follows from (\ref{Cprobbnd1}).
\end{proof}

\medskip

Now let us prove a simple lemma that will allow us to choose the ``right" 
$\xi$ that will satisfy the conditions of the previous corollary and allow us to push through the 
uniqueness argument. One inequality below is needed to make Corollary~\ref{cor:ubnd} apply. 
The other is required for the proof of Lemma~\ref{I3lemma}. 
\begin{lemma}
\label{lemmaonrightxi}
Fix $\alpha, \gamma$ satisfying the conditions of
Theorem~\ref{thm:uniquekunbounded}, that is, 
\begin{eqnarray}
\label{equt:25}
1>\gamma>\frac{1}{2(1-\alpha)}>\frac12.
\end{eqnarray}
Then we  we can choose $\xi\in (0,1)$ such that 
\begin{eqnarray}\label{alphbnd}
\frac{\alpha}{2\gamma-1}<\xi<\left(\frac{\frac12-\alpha}{1-\gamma}\wedge 1\right).
\end{eqnarray}
\end{lemma}
\begin{proof}
Let us verify that~(\ref{alphbnd}) is possible. There are two cases, the first being
\begin{eqnarray}
\label{equt:27}
\frac{\frac12-\alpha}{1-\gamma}<1.
\end{eqnarray}
Recall that $\alpha\in (0,\frac12)$ and $\gamma\in(\frac12,1)$. Therefore
\begin{eqnarray}
\nonumber
\frac{\frac12-\alpha}{1-\gamma}-\frac{\alpha}{2\gamma-1}&=&
\frac{(\frac12-\alpha)(2\gamma-1)-\alpha(1-\gamma)}{(1-\gamma)(2\gamma-1)}\\
\nonumber
&=&\frac{\gamma(1-\alpha) -\frac12}{(1-\gamma)(2\gamma-1)}
\\
\label{equt:26}
&>&0,
\end{eqnarray}
where the last inequality follows by~(\ref{equt:25}). Then (\ref{equt:26})
implies that we can fix $\xi$ satisfying~(\ref{alphbnd}) in the case of~(\ref{equt:27}).  

The second case is for
$$
\label{equt:28}
\frac{\frac12-\alpha}{1-\gamma}\geq 1, 
$$
that is 
$$
\label{equt:29}
\alpha \leq \gamma-\frac12.
$$
Hence 
$$
\label{equt:30}
\frac{\alpha}{2\gamma-1}\leq \frac12, 
$$
and we can easily fix  $\xi$ satisfying~(\ref{alphbnd}) in this case as well.

\end{proof}

Now fix $\xi$ as in the previous lemma and define 
\begin{eqnarray*}
\eta\equiv\frac{\xi}{\alpha}.
\end{eqnarray*}
Lemma~\ref{lemmaonrightxi} immediately implies that
\begin{eqnarray}
\label{equt:31}
\eta>\frac{1}{2\gamma-1}.
\end{eqnarray}

We return to the setting and notation of Section 3. In particular $\Psi\in
C_c^\infty([0,t]\times\IR)$ with $\Gamma(t)=\{x:\exists s\le t,\,\Psi_s(x)>0
\}\subset B(0,J(t))$.  

For $a_n$ given by~\eqref{acond}, let $m^{(n)}:=a_{n-1}^{-\frac{1}{\eta}}$. Note that $m^{(n)}\ge
1$ for all $n$.
Set $c_{0}(K):= r_0(K) \vee K^{2}$ (where $r_0(K)$ is chosen as in Corollary \ref{cor:ubnd})
and define the stopping time
\begin{eqnarray*}
& &T_{\xi,K}= \inf\{ t \geq 0: t>T_{K} \text { or }t \leq T_{K} \text{ and there exist }
\epsilon \in (0,1], \hat{x},x,y \in \IR \text{ with } \\
& & \phantom{AAAA}|x|\le \ep^{\alpha}, |\tilde{X}(t,\hat{x})| \leq \epsilon^{\xi},
|x-\hat{x}| \leq \epsilon^{\alpha}, |x-y|\leq \epsilon^{\alpha} \text{ such that } |\tilde{X}(t,y)|> c_{0}(K)
\epsilon^{\xi} \}.
\end{eqnarray*}
Assuming our filtration is completed as usual, $T_{\xi,K}$ is a stopping time 
by the standard projection argument.
Note that for any $t\geq 0,$ by Corollary~\ref{cor:ubnd},
\begin{eqnarray}
\nonumber
\IP( T_{\xi,K} \leq t) &\leq& \IP( T_K \leq t) + \IP(C_{\xi,K} >c_0(K)) \\
\label{eq:stoptbnd}
&\leq& \IP( T_K \leq t)
+ c_{\ref{Cprobbnd}.1}\Bigl[ \Bigl(\frac{K^{2}-6}{K+1
}\Bigr)^{-\delta}\\
\nonumber
& &\phantom{\IP( T_{\xi,K} \leq t) \leq  }
+K \exp\Bigl(-c_{\ref{Cprobbnd}.2}\Bigl({K^{2}-6\over K+1}\Bigr)^\delta\Bigr)\Bigr]
\end{eqnarray}
which tends to zero as $K \rightarrow \infty$ 
due to~(\ref{equt:120}). 

With this set-up we can show the following lemma:
\begin{lemma}
\label{lem:ubnd}
For all $x \in B(0, \frac{1}{m^{(n)}})$ and $s\in[0,T_{\xi,K}]$, if
$ |\la \tilde{X}_{s}, \Phi_{x}^{m^{(n)}} \ra| \leq a_{n-1}$ then
\begin{equation*}
\sup_{y \in B(x, \frac{1}{m^{(n)}})} |\tilde{X}(s,y)| \leq c_{0}(K) a_{n-1}.
\end{equation*}
\end{lemma}
\begin{proof}
Since $ |\la \tilde{X}_{s}, \Phi_{x}^{m^{(n)}} \ra| \leq a_{n-1}$ and $\tilde
X_s(\cdot)$ is continuous there exists an $\hat{x} \in B(x, \frac{1}{m^{(n)}})$ such
that
$|\tilde{X}(s,\hat{x})| \leq a_{n-1}.$ 
Apply the definition of the stopping time
with $\epsilon^\alpha=1/m^{(n)}\in (0,1]$ and so $\epsilon^\xi=a_{n-1}$ to obtain the required
bound.
\end{proof}

\smallskip

Next, we bound
$|I_{3}^{m^{(n)},n}|$ of~\eqref{eq:Iparts} using the H\"older continuity of
$\sigma,$ as well  as the definition of $\psi_{n}$.  

\begin{lemma}
\label{I3lemma}
\begin{equation}
\label{eq:limI_3}
\lim_{n \rightarrow \infty} \IE \Big( |I_{3}^{m^{(n)},n}(t\wedge T_{\xi,K})| \Big)=0
\end{equation}
\end{lemma}
\begin{proof}
By~\eqref{eq:sigma_Hold} we have
\begin{eqnarray*}
|I_{3}^{m^{(n)},n}(t \wedge T_{\xi,K})|
&\leq& \frac{L^2}{n}
\int_{0}^{t\wedge T_{\xi,K}} \int_{\IR}
1_{\{a_{n} \leq |\la \tilde{X}_{s}, \Phi_{x}^{m^{(n)}} \ra| \leq a_{n-1}\} } \,
a_{n}^{-1} |\tilde{X}(s,0)|^{2\gamma}  \\
& &\phantom{AAAAAAAAA}\times
\Phi_{x}^{m^{(n)}}(0)^2 \Psi_s(x) dx
ds.
\end{eqnarray*}
We obtain from Lemma~\ref{lem:ubnd}
\begin{align*}
&|I^{m^{(n)},n}_{3}(t \wedge T_{\xi,K})| \\
&\qquad\leq  L^2c_{0}(K)^{\gamma} \frac{a_{n-1}^{2
\gamma}}{2na_{n}}
\int_{0}^{t\wedge T_{\xi,K}} \int_{\IR}
1_{\{a_{n} \leq |\la \tilde{X}_{s}, \Phi_{x}^{m^{(n)}} \ra| \leq a_{n-1}\} } \,\\
& \qquad\qquad\qquad\qquad\qquad\qquad\qquad\qquad\qquad
\cdot \Phi_{x}^{m^{(n)}}(0)^2 \Psi_s(x) dx ds\\
&\qquad\leq \frac{ L^2 
||\Psi||_{\infty} c_{0}(K)^{2\gamma}}{n} \frac{a_{n-1}^{2 \gamma}}{a_{n}}
\int_{0}^{t\wedge T_{\xi,K}} 
\left( \int_{\Gamma(t)}
\Phi_{x}^{m^{(n)}}(0)^2 dx \right) ds\\
&\qquad\leq \frac{ L^2 
||\Psi||_{\infty} c_{0}(K)^{2\gamma}t}{n} \frac{a_{n-1}^{2 \gamma}}{a_{n}}
m^{(n)}c_\theta\\
&\qquad\leq \frac{C(L, \Psi) c_{0}(K)^{2\gamma} t}{n} 
\frac{a_{n-1}^{2
\gamma}}{a_{n}} a_{n-1}^{-1/\eta}\\
&\qquad= \frac{C(L, \Psi)  c_{0}(K)^{2\gamma}t}{n}
\frac{a_{n-1}^{(2\gamma-\frac{1}{\eta} )} }{a_{n}}.
\end{align*}

\noindent
If we choose $\rho(x)=\sqrt{x}$ then
$\int_{a_{n}}^{a_{n-1}} x^{-1} dx = n$
so that $\frac{a_{n-1}}{a_{n}} = e^{n}$ or (using that $a_{0}=1$)
$a_{n} = e^{-\frac{n(n+1)}{2}}.$ Thus~\eqref{eq:limI_3} holds
if
$n(n+1) - (2\gamma-\frac{1}{\eta} )(n-1)n <0$
for $n$ large. This is equivalent to
\begin{equation*}
1-(2\gamma-\frac{1}{\eta	} ) <0 \Leftrightarrow
\gamma > \frac{1}{2} + \frac{1}{2\eta}
\end{equation*}
which holds by (\ref{equt:31}). A similar argument applies for any $\rho$ satisfying~\eqref{rhobnd} and~\eqref{rhocond}.
\end{proof}

Use (\ref{absconv}) and
Fatou's Lemma on the left-hand side of
(\ref{eq:Iparts}),
and Lemmas~\ref{lem:conv} and~\ref{I3lemma} on the
right-hand side, to take limits in this equation and
so conclude that
\begin{eqnarray*}
\int_{\IR} \IE\Big( |\tilde{X}(t\wedge T_{\xi,K},x)|\Big) \Psi_t(x) dx
&\leq&  \liminf_{n \rightarrow \infty} \int_{\IR} \IE \Big(
\phi_{n}(\la \tilde{X}_{t \wedge T_{\xi,K}}, \Phi^{m^{(n)}}_{x} \ra) \Big)
\Psi_t(x) dx \\
&\leq&  \IE \Big( \int_{0}^{t \wedge T_{\xi,K}} \int_{\IR}
|\tilde{X}(s,x)|
\Bigl(\Delta_{\theta} \Psi_s(x)+\dot\Psi_s(x)\Bigr)dx ds \Big)\\
&\leq& \int_{0}^{t} \int_{\IR} \IE \Big(  |\tilde{X}(s,x)| \Big)
|\Delta_{\theta} \Psi_s(x)+\dot\Psi_s(x)|dx ds.
\end{eqnarray*}
Since $T_{\xi,K}$ tends in probability to infinity as $K \rightarrow \infty$
according to (\ref{eq:stoptbnd}), we have that \hfil\break
$\tilde{X}(t\wedge
T_{\xi,K},x)
\rightarrow \tilde{X}(t,x)$ and so we finally conclude with another application of
Fatou's Lemma that
\begin{equation}
\label{eq:heatGron}
\int_{\IR} \IE\Big( |\tilde{X}(t,x)|\Big) \Psi_t(x) dx
\leq \int_{0}^{t} \int_{\IR} \IE \Big(  |\tilde{X}(s,x)| \Big)
\Bigl|
\Delta_{\theta} \Psi_s(x)+\dot\Psi_s(x)\Bigr|dx ds.
\end{equation}

Let $\{g_N\}$ be a sequence of functions in $C_c^\infty(\IR)$ such that
$g_N:\IR\to[0,1]$,
$$B(0,N)\subset\{x:g_N(x)=1\},\quad B(0,N+1)^c\subset\{x:g_N(x)=0\},$$
and
$$\sup_{N\geq 1}[\Vert |x|^{-\theta} g'_N\Vert_\infty+\Vert \Delta_{\theta}g_N\Vert_\infty]\equiv
C<\infty,$$
where $g'_N$ denotes the derivative with respect to the spatial variable.
Now let $\phi\in C_c^\infty(\IR)$, and for $(s,x)\in[0,t]\times\IR$
set $\Psi_N(s,x)=(S_{t-s}\phi(x)) g_N(x)$.  It is then easy to check that $\Psi_N\in
C_c^\infty([0,t]\times \IR)$ and for $\lambda>0$ there is a
$C=C(\lambda,\phi,t)$ such that for all $N$
\begin{eqnarray*}
|\Delta_{\theta}
\Psi_N(s,x)+\dot\Psi_N(s,x)|&=&
\left| 4\alpha^2|x|^{-\theta}\frac{\partial}{\partial x}
S_{t-s}\phi(x)\frac{\partial}{\partial x}g_N(x)
+S_{t-s}\phi(x)\Delta_{\theta}g_N(x)\right|\\
&\le&C e^{-\lambda|x|}1_{\{|x|>N\}}.
\end{eqnarray*}
Use this in (\ref{eq:heatGron}) to conclude that
$$\int_{\IR}\IE(|\tilde X(t,x)|)\phi(x)\,dx\le
C\int_0^t\int_{\IR}\IE(|\tilde X(s,x)|)e^{-\lambda
|x|}1_{\{|x|>N\}}\,dx\,ds.$$
By Proposition~\ref{prop} 
the right-hand side of the above approaches zero as $N\to\infty$ and we
see that
$$\IE\Bigl(\int_{\IR}|\tilde X(t,x)|dx\Bigr)=0.$$
Therefore $X^1(t)=X^2(t)$ for all $t\ge 0$ a.s. by continuity.
\gdm

\bigskip

\section{H\"older continuity:  Proposition~\ref{prop}}
\label{sec:prop}
First we will introduce a number of technical lemmas that will be frequently used. 
The proof of the next lemma is elementary and therefore is omitted. 
\begin{lemma}
\label{lem:2}
For any $x,y\in \IR$, $0\leq\beta\leq 1$,
\begin{equation*}
 |p_t(x)-p_t(y)|\leq ct^{-\alpha}\left(\frac{|x-y|}{t}\right)^{\beta}\left(\max(|x|,|y|)\right)^{(\frac{1}{\alpha}-1)\beta}.
\end{equation*}
\end{lemma}

\begin{lemma}
\label{lem:23_1}
For any $0<t<t'\leq T$, $x\in\IR$,
\begin{equation*}
\int_0^t (p_{t'-s}(x)-p_{t-s}(x))^2\,ds \leq c(T)|t'-t|^{1 -2\alpha}. 
\end{equation*}
\end{lemma}
\begin{proof}
Assume, without loss of generality, that $t'-t\leq t$. 
\begin{eqnarray}
\label{equt:23_1}
 \int_0^t (p_{t'-s}(x)-p_{t-s}(x))^2\,ds &\leq& 
c \int_{t-|t'-t|}^{t} (t-s)^{-2\alpha}\,ds
+ \int_{0}^{t-|t'-t|} (p_{t'-s}(x)-p_{t-s}(x))^2\,ds
\\
\nonumber
&\leq& c|t'-t|^{1-2\alpha}+ 
c\int_0^{t-|t'-t|}  \left| ((t-s)^{-\alpha}-(t'-s)^{-\alpha})e^{-\frac{|x|^{1/\alpha}}{t-s}}\right|^2\,ds\\
\nonumber
&&\mbox{}+
c \int_{0}^{t-|t'-t|} \left| (t'-s)^{-\alpha}(e^{-\frac{|x|^{1/\alpha}}{t-s}}-e^{-\frac{|x|^{1/\alpha}}{t'-s}})\right|^2\,ds
\end{eqnarray}
The second term on the right hand side is trivially bounded by
\begin{eqnarray*}
 c\int_0^{t-|t'-t|}  ((t-s)^{-\alpha}-(t'-s)^{-\alpha})^2\,ds
&\leq& c\int_0^{t} (t-s)^{-2\alpha}-(t'-s)^{-2\alpha}\,ds\\
&\leq&  c|t'-t|^{1-2\alpha}.
\end{eqnarray*}
The third term on the right hand side of~\eqref{equt:23_1} is also easy to bound, as
\begin{eqnarray*}
c \int_{0}^{t-|t'-t|} (t'-s)^{-2\alpha}e^{-\frac{2|x|^{1/\alpha}}{t'-s}} \frac{|x|^{1/\alpha}(t'-t)}{(t-s)(t'-s)}\,ds
&\leq& c (t'-t) \int_{0}^{t-|t'-t|} (t'-s)^{-2\alpha} (t-s)^{-1}\,ds\\
&\leq&  c|t'-t|^{1-2\alpha}.
\end{eqnarray*}
This completes the proof.
\end{proof}
\begin{lemma}
\label{lem:23_2}
For any $x,y\in [-1,1]$, $t\leq T$, $\beta\in (1/2-\alpha)$,
\begin{equation*}
\int_0^t (p_{t-s}(x)-p_{t-s}(y))^2\,ds \leq c(T)\left(\max(|x|,|y|)\right)^{(\frac{1}{\alpha}-1)2\beta}|x-y|^{1-2\alpha}. 
\end{equation*}
\end{lemma}
\begin{proof}
\begin{equation}
\label{equt:23_2}
\int_0^t (p_{t-s}(x)-p_{t-s}(y))^2\,ds \le \int_0^{t-|x-y|} (p_{t-s}(x)-p_{t-s}(y))^2\,ds
+ \int_{t-|x-y|}^t (p_{t-s}(x)-p_{t-s}(y))^2\,ds,
\end{equation}
so by Lemma~\ref{lem:2} (by taking $\beta=1$ there) we can bound the first term on the right hand side 
by 
\begin{eqnarray*}
\lefteqn{\left(\max(|x|,|y|)\right)^{(\frac{1}{\alpha}-1)2}\int_0^{t-|x-y|} (t-s)^{-2\alpha-2} |x-y|^2\,ds}
\\
&\leq& C \left(\max(|x|,|y|)\right)^{(\frac{1}{\alpha}-1)2\beta} |x-y|^{-2\alpha-1+2}\\
&\leq&\left(\max(|x|,|y|)\right)^{(\frac{1}{\alpha}-1)2\beta} C |x-y|^{-2\alpha+1}.   
\end{eqnarray*}
By Lemma~\ref{lem:2} again, with $\beta$ as given in this lemma, we can bound the second  term on the right hand side 
of~\eqref{equt:23_2} 
by 
\begin{eqnarray*}
\lefteqn{\left(\max(|x|,|y|)\right)^{(\frac{1}{\alpha}-1)2\beta} |x-y|^{2\beta}
\int_{t-|x-y|}^t (t-s)^{-2\alpha-2\beta}\,ds}
\\
&\leq&  C \left(\max(|x|,|y|)\right)^{(\frac{1}{\alpha}-1)2\beta}|x-y|^{-2\alpha+1},  
\end{eqnarray*}
and we are done.
\end{proof}

\paragraph{Proof of Proposition~\ref{prop}}
\begin{itemize}
\item[{\bf (a)}]
follows by the correspondence between the SPDE~(\ref{equt:4}) and SIE~\eqref{equt:21_1} and the moment bound 
(\ref{equt:23_3}). 

\item[{\bf (b)}]
Let $Z(t,x)=  X(t,x)-S_tX_0(x)-\int_0^t \int_{\Resm}p^{\theta}_{t-s}(x,y) g(s,y)\,dy\,ds$. Then
\begin{eqnarray*}
|Z(t',x)-Z(t,y)|&=& \int_0^t (p^\theta_{t'-s}(x)-p^\theta_{t-s}(y))\sigma(X(s,0))\,dB_s
\\
&&\mbox{}+
 \int_t^{t'} p^\theta_{t'-s}(x)\sigma(X(s,0))\,dB_s\\
&=& \int_0^t (p^\theta_{t'-s}(x)-p^\theta_{t-s}(x))\sigma(X(s,0))\,dB_s\\
&&\mbox{}+
\int_0^t (p^\theta_{t-s}(x)-p^\theta_{t-s}(y))\sigma(X(s,0))\,dB_s\\
&&\mbox{}+
 \int_t^{t'} p^\theta_{t'-s}(x)\sigma(X(s,0))\,dB_s\,.
\end{eqnarray*}
By the Burkholder-Gundy-Davis and H\"older inequalities, the moment bound on $X(s,0)$,
  and Lemmas~\ref{lem:23_1} and~\ref{lem:23_2}, the required moment bound
\[ \IE\left[|Z(t',x)-Z(t,y)|^p\right]\leq C(|t'-t|^{p(1/2-\alpha)}+|x-y|^{p(1/2-\alpha)})\]
now follows. By the Kolmogorov criterion we also get the required H\"older continuity of $Z$.
\end{itemize}
\gdm\nopagebreak

\section{H\"older continuity: Theorem~\ref{thm:H1}}
\label{sec:proofofimprovement}

{\bf Proof of Theorem~\ref{thm:H1}} 

We proceed along the lines of the proof of Theorem~4.1 of~\cite{bib:msp05}.
Fix arbitrary (deterministic) $(t,x), (t',y)$ such that
$|t-t'|\le \epsilon\equiv 2^{-N}$ ($N\in\IN$), $\vert x\vert\leq 2^{-N\alpha}$,
 $\vert x-y\vert\leq 2^{-N\alpha}$ and $t\le t'$ (the 
case $t' \leq t$
works analogously).

In the following we will define small numbers $\delta,\delta',\delta_1\,,\delta_2>0$ as follows. 
As $\xi_1<(\xi\gamma+\frac{1}{2}-\alpha)\wedge1$, we may choose
$\delta\in(0,\frac{1}{2}-\alpha)$ such that 
\begin{equation*}
 \xi_1< ((\xi\gamma+\frac{1}{2}-\alpha)\wedge1)-\alpha\delta<1.
\end{equation*} 
Fix $\delta'\in (0,\delta)$. Now choose $\delta_1\in (0,\delta')$ sufficiently small
such that 
\begin{equation}
\label{equt:11}
 \xi_1< ((\xi\gamma+\frac{1}{2}-\alpha)\wedge1)-\alpha\delta+\alpha\delta_1<1.
\end{equation} 
Moreover define 
\begin{eqnarray*}
p\equiv  ((\xi\gamma+\frac{1}{2}-\alpha)\wedge1)-\alpha(1/2-\alpha)+\alpha\delta_1\,,
\end{eqnarray*} 
and hence by~(\ref{equt:11}) we easily get that 
\begin{eqnarray}
\label{equt:10}
 p+\alpha(1/2-\alpha-\delta)=((\xi\gamma+\frac{1}{2}-\alpha)\wedge1)-\alpha\delta+\alpha\delta_1
\in (\xi_1\,,1). 
\end{eqnarray} 
Also define $\delta_2>0$ sufficiently small such that
\begin{eqnarray}
\label{equt:12}
\delta'-\delta_2>\delta_1\,
\end{eqnarray} 
 and define 
\begin{eqnarray}
\label{equt:17}
\nonumber
 \hat p &\equiv& p+ \alpha(\delta'-\delta_2-\delta_1)\\
\label{equt:18}
&=&((\xi\gamma+\frac{1}{2}-\alpha)\wedge1)-\alpha(1/2-\alpha)+\alpha(\delta'-\delta_2).
\end{eqnarray} 
By~(\ref{equt:12}) we get that 
\[\hat p > p.\]




\noindent Now consider for some random $N_1=N_1(\omega,\xi,\xi_1)$ (to be chosen below in~\eqref{equt:24}),
\begin{eqnarray}
\label{eq:H1}
\lefteqn{\IP\left( |\tilde X(t,x)-\tilde X(t,y)| \geq
|x-y|^{\frac{1}{2}-\alpha-\delta}\epsilon^{p}, (t,x)\in Z_{K,N,\xi}, 
N\ge N_1
 \right)}\\
\nonumber
&&\mbox{}+\IP\left( |\tilde X(t',x)-\tilde X(t,x)| \geq
|t'-t|^{\alpha(\frac{1}{2}-\alpha-\delta)}\epsilon^{p}, (t,x)\in Z_{K,N,\xi}, t'\le T_K,
N\ge N_1
 \right).
\end{eqnarray}
In what follows we are going to obtain the bound on~\eqref{eq:H1}. 
Set
\begin{eqnarray*}
D^{x,y,t,t'}(s)&=&
\left| p_{t-s}(x)- p_{t'-s}(y)\right|^2
|\wX(s,0)|^{2\gamma},\\
D^{x,t'}(s)&=& p_{t'-s}(x)^2
|\wX(s,0)|^{2\gamma}.
\end{eqnarray*}
With this notation, expression (\ref{eq:H1}) is bounded by
\begin{eqnarray}
\label{eq:H3}
& &\IP\Big( |\wX(t,x)-\wX(t,y)|
\geq |x-y|^{\frac{1}{2}-\alpha-\delta}\epsilon^{p}, (t,x)\in Z_{K,N,\xi},N \geq N_1\\
\nonumber
& &\phantom{AAAAAAAAAAAAAAAAAAAAAA}
\int_{0}^{t}  D^{x,y,t,t}(s)  ds
\leq  |x-y|^{1-2\alpha-2\delta'}\epsilon^{2p} \Big)\\
\nonumber
&+&\IP\Big( |\wX(t',x)-\wX(t,x)|
\geq |t'-t|^{\alpha(\frac{1}{2}-\alpha-\delta)}\epsilon^{p}, (t,x)\in Z_{K,N,\xi}, t'\le
T_K, N \geq N_1\\
\nonumber
& &\phantom{A}
\int_{t}^{t'}  D^{x,t'}(s) ds
+ \int_{0}^{t}  D^{x,x,t,t'}(s) ds \leq
(t'-t)^{2\alpha(\frac{1}{2}-\alpha-\delta')}\epsilon^{2p}
\Big)\\
\nonumber
&+& \IP\Big( \int_{0}^{t} 
D^{x,y,t,t}(s) ds
> |x-y|^{1-2\alpha-2\delta'}\epsilon^{2p}, (t,x)\in
Z_{K,N,\xi}, N\ge N_1
\Big)\\
\nonumber
&+& \IP\Big( 
\int_t^{t'}D^{x,t'}(s) ds
+ \int_{0}^{t}  D^{x,x,t,t'}(s)  ds\\
\nonumber
& &\phantom{AAAAAAAAAA}
>  (t'-t)^{2\alpha(\frac{1}{2}-\alpha-\delta')}\epsilon^{2p},
 (t,x)\in Z_{K,N,\xi}, t'\le T_K, N\ge N_1 \Big)\\
\nonumber
&=:& P_{1}+P_{2}+P_{3}+P_{4}.
\end{eqnarray}

\bigskip

Notice that the processes 
\[\tilde{t} \mapsto \int_{0}^{\tilde{t}}
 ( p_{t-s}(x))
\left(\sigma(X^{1}(s,0))-\sigma(X^{2}(s,0)\right) B(ds)\]
are continuous local martingales for any fixed $x,t$ on $0 \leq \tilde{t} \leq t$. 
We bound the appropriate differences of these integrals by considering the respective
quadratic variations of $\wX(t,x)-\wX(t,y)$  and  $\wX(t',x)-\wX(t,x)$
 (see (\ref{equt:4})). By~\eqref{eq:sigma_Hold},
we see that the time integrals in the above
probabilities differ from the appropriate square functions by a multiplicative factor of
$L^2 $. 

If $\delta''=\delta-\delta'>0$, $B$ is a standard one-dimensional Brownian motion with
$B(0)=0$, and $B^{*}(t):=\sup_{0\leq s \leq t} |B(s)|,$ then
$P_1$  of (\ref{eq:H3}) can be bounded
using the Dubins-Schwarz Theorem:
\begin{eqnarray}
\label{P1est}
\nonumber
P_{1}&\leq& \IP\left(
B^{*}(L^2|x-y|^{1-2\alpha-2\delta'}\epsilon^{2p})
\geq |x-y|^{\frac{1}{2}-\alpha-\delta}\epsilon^{p}\right)\\
\nonumber
&=& \IP\left(
B^{*}(1)L |x-y|^{\frac{1}{2}-\alpha-\delta'}\epsilon^{p}
\geq |x-y|^{\frac{1}{2}-\alpha-\delta}\epsilon^{p}\right)
\\
\label{equt:19}
&=& \IP\left( B^{*}(1)
\geq  L^{-1}|x-y|^{-\delta''}  \right)
\leq c_{\ref{P1est}} \exp(-c'_{\ref{P1est}} |x-y|^{-\delta''}),
\end{eqnarray}
where we have used the reflection principle in the last line. Similarly,
\begin{eqnarray}
\label{P2est}
\nonumber
P_{2}&\leq& \IP\left(
B^{*}(L^2|t'-t|^{1-2\alpha-\delta'}\epsilon^{2p})
\geq |t'-t|^{\frac{1}{2}-\alpha-\delta/2}\epsilon^{p}\right)\\
\nonumber
&=& \IP\left(
B^{*}(1)L |t'-t|^{2\alpha(\frac{1}{2}-\alpha-\delta')}\epsilon^{p}
\geq |t'-t|^{\alpha(\frac{1}{2}-\alpha-\delta)}\epsilon^{p}\right)
\\
\label{equt:20}
&=& \IP\left( B^{*}(1)
\geq  L^{-1}|t'-t|^{-\alpha\delta'' } \right)
\leq c_{\ref{P1est}} \exp(-c'_{\ref{P1est}} |t'-t|^{-\alpha\delta''}),
\end{eqnarray}
Here the constants $c_{\ref{P1est}}$  and
$c'_{\ref{P1est}}$ depend on $L$. 


Before we proceed with bounds on $P_3, P_4$, in the next lemma, we will obtain a useful bound on $\tilde{X}(s,0)$. 


\begin{lemma}
\label{tubnd}
Let $N\ge N_{\xi}.$ Then on $\{\omega:(t,x)\in Z_{K,N,\xi}\}$,
\begin{eqnarray}
\label{equt:tubnd}
|\tilde X(s,0)| &\le& 3\ep^{\xi}
\quad \quad \quad \text{ for } s \in [t-\epsilon,t'], \\
\label{1.1}
|\tilde X(s,0)| &\leq&  (4+ K)2^{\xi N_{\xi}}(t-s)^{\xi}
\quad \quad \text{ for } s \in [0,t-\epsilon].
\end{eqnarray}
\end{lemma}
\begin{proof}

\noindent
Assume $(t,x)\in Z_{K,N,\xi}$, $0\le t'\le T_K$ and choose $(\hat{t},\hat{x})$ such that
\begin{eqnarray*}
\label{t0x0}
\hbox{
$\hat{t}\le T_K$, $|t-\hat{t}|\leq \epsilon=2^{-N}, |\hat{x}-x|\leq \epsilon^{\alpha}$, and $|\tilde X(\hat{t},\hat{x})|\le
2^{-N\xi}=\epsilon^\xi$.}
\end{eqnarray*}
We first observe that for $s \in [t-\epsilon,t']$, we trivially have $|t-s|\leq \epsilon$.
Therefore by (\ref{reghyp}) and the definition of $Z_{K,N,\xi}$, for $s \in [t-\epsilon,t']$ we get 
\begin{eqnarray}
\nonumber
|\tilde X(s,0)|&\le & |\tilde X(\hat{t},\hat{x})|+|\tilde X(\hat{t},\hat{x})-\tilde X(t,x)|+
 |\tilde X(t,x)-\tilde X(s,0)|\\
\nonumber
&\le &  3 \cdot 2^{-N\xi}\\
\label{tubnd1}
&= &3\epsilon^{\xi},
\end{eqnarray}
which proves (\ref{equt:tubnd}).

If $s \in [t-2^{-N_{\xi}},t-\epsilon]$, then there exists $\tilde{N}\geq N_{\xi}$ such that  $2^{-(\tilde{N}+1)}\leq
 t-s\leq 2^{-\tilde{N}}$ so that as in (\ref{tubnd1}) we can bound 
\begin{eqnarray*}
\nonumber
|\tilde X(s,0)|&\le & |\tilde X(\hat{t},\hat{x})|+|\tilde X(\hat{t},\hat{x})-\tilde X(t,x)|+
 |\tilde X(t,x)-\tilde X(s,0)|\\
\nonumber
&\le &  2^{-N\xi}+2^{-N\xi}+2^{\xi} \cdot 2^{-(\tilde{N}+1)\xi}\\
\nonumber
&\leq& 2\cdot (t-s)^{\xi}+ 2 \cdot (t-s)^{\xi}\\
\label{tubnd2}
&= &4(t-s)^{\xi},
\end{eqnarray*}
which proves~(\ref{1.1}) for $s \in [t-2^{-N_{\xi}},t-\epsilon]$. For $s \in [0,t-2^{-N_{\xi}}]$ 
we bound
\begin{eqnarray}
\nonumber
|\tilde X(s,0)|&\le & K\\
\nonumber
&\le &  K(t-s)^{-\xi}(t-s)^{\xi}\\
\nonumber
&\leq& K2^{N_{\xi}\xi}(t-s)^{\xi},
\end{eqnarray}
and we are done. 

\end{proof}

For the rest of  this section $C(K)$ will be a constant depending on $K$ which may change from line to line. The next lemma is crucial for bounding 
 $P_3$. 
\begin{lemma}
\label{lem:7.2}
Let $N\ge N_{\xi}.$ Then on $\{\omega:(t,x)\in Z_{K,N,\xi}\}$,
 \[ \int_{0}^{t} 
D^{x,y,t,t}(s) ds
\leq C(K)2^{2\gamma\xi N_{\xi}}\ep^{2\hat p} |x-y|^{1-2\alpha-2\delta'}.
\]
\end{lemma}

\begin{proof} First we split the integral: 
\begin{eqnarray*}
\int_{0}^{t} 
D^{x,y,t,t}(s) ds&=& \int_{t-\epsilon}^{t} 
D^{x,y,t,t}(s) ds+\int_{0}^{t-\epsilon} 
D^{x,y,t,t}(s) ds \\
&=:& D_1(t)+D_2(t).
\end{eqnarray*}
By Lemma~\ref{tubnd} we get 
\begin{eqnarray*}
D_1(t)&\leq& \int_{t-\ep}^t (p_{t-s}(x)-p_{t-s}(y))^2 \ep^{2\xi\gamma}\,ds. 
\end{eqnarray*}
Now apply Lemma~\ref{lem:23_2} with $\beta=1/2-\alpha-\delta'$ to get 
\begin{eqnarray*}
D_1(t)&\leq&  c \ep^{2\xi\gamma}|x-y|^{1-2\alpha}
 \left(\max(|x|,|y|)\right)^{(1/\alpha-1)2\beta}.
\end{eqnarray*}
Now recall that 
\begin{eqnarray}
\label{equt:8}
 \max(|x|,|y|)\leq c\ep^{\alpha}
\end{eqnarray}
 and we get 
\begin{eqnarray*}
D_1(t)&\leq&  c \ep^{2\xi\gamma} \ep^{2\delta'} |x-y|^{1-2\alpha-2\delta'}
 \ep^{(1-\alpha)2\beta}\\
&=&c\ep^{2(1/2 - \alpha(3/2-\alpha)+\alpha\delta'+\gamma\xi)}|x-y|^{1-2\alpha-2\delta'}\\
&\leq& c\ep^{2\hat{p}}|x-y|^{1-2\alpha-2\delta'},
\end{eqnarray*}
where the last line follows by~(\ref{equt:18}). 

Now we will bound $D_2(t)$.  
By Lemma~\ref{tubnd} we get 
\begin{eqnarray*}
D_2(t)&\leq& C(K)2^{2\gamma\xi N_{\xi}} \int_{0}^{t-\ep} (p_{t-s}(x)-p_{t-s}(y))^2 (t-s)^{\xi}\,ds. 
\end{eqnarray*}
Apply Lemma~\ref{lem:2} with $\beta=1$ and use~(\ref{equt:8}) to get 
\begin{eqnarray*}
D_2(t)&\leq&  C(K)2^{2\gamma\xi N_{\xi}} \int_0^{t-\ep} (t-s)^{-2\alpha-2+2\gamma\xi}|x-y|^{2}
\ep^{2(1-\alpha)} \,ds\\
&=&C(K,\delta,\delta_2)2^{2\gamma\xi N_{\xi}} \ep^{((-2\alpha-1+2\gamma\xi)\wedge 0)-2\alpha\delta_2} \ep^{2(1-\alpha)} 
 |x-y|^{1-2\alpha-2\delta'}|x-y|^{1+2\alpha+2\delta'}\\
&\leq& C(K,\delta,\delta_2)2^{2\gamma\xi N_{\xi}} \ep^{((-2\alpha-1+2\gamma\xi)\wedge 0)-2\alpha\delta_2}\ep^{2(1-\alpha)} 
 |x-y|^{1-2\alpha-2\delta'}\ep^{\alpha(1+2\alpha+2\delta')}\\
&=&C(K,\delta,\delta_2)2^{2\gamma\xi N_{\xi}}\ep^{2\hat p} |x-y|^{1-2\alpha-2\delta'},
\end{eqnarray*}
where the last equality follows easily by the simple algebra and the definition of $\hat p$. 

\end{proof}

The next lemma is important  for bounding 
 $P_4$.
\begin{lemma}
\label{lem:7.3}
Let $N\ge N_{\xi}.$ Then on $\{\omega:(t,x)\in Z_{K,N,\xi}\}$,
\begin{eqnarray}
\label{eq:7.3.1}
\lefteqn{\int_t^{t'}D^{x,t'}(s) ds
+ \int_{0}^{t}  D^{x,x,t,t'}(s)  ds}\\
\nonumber
&\leq& C(K)2^{2\gamma\xi N_{\xi}}\ep^{2\hat p} |t'-t|^{\alpha(1-2\alpha-2\delta')}. 
\end{eqnarray}
\end{lemma}
\begin{proof}
By Lemma~\ref{tubnd} we have,
\begin{eqnarray}
\nonumber
\int_t^{t'}D^{x,t'}(s) ds &=& \int_t^{t'} p_{t'-s}(x)^2 |\widetilde{X}(s,0)|^{2\gamma}\,ds\\
\nonumber
&\leq&  c\int_t^{t'} p_{t'-s}(0)^2 \ep^{2\xi\gamma}\,ds\\
\nonumber
&=&c\ep^{2\xi\gamma}\int_t^{t'} (t'-s)^{-2\alpha}\,ds\\
\nonumber
&=&c\ep^{2\xi\gamma}|t'-t|^{1-2\alpha}
\end{eqnarray}
As for the second term at the left hand side of~\eqref{eq:7.3.1}, we first split it: 
\begin{eqnarray*}
\int_{0}^{t}  D^{x,x,t,t'}(s)  ds&=& \int_{t-\ep}^{t}  D^{x,x,t,t'}(s)\,ds +\int_{0}^{t-\ep}  D^{x,x,t,t'}(s)\,ds\\
&=:& D_1(t)+D_2(t).
\end{eqnarray*}
Then by Lemma~\ref{lem:23_1} and \eqref{tubnd1} we have 
\begin{eqnarray}
\nonumber
D_1(t)&=& \int_{t-\ep}^{t}  \left| p_{t-s}(x)- p_{t'-s}(x)\right|^2|\widetilde{X}(s,0)|^{2\gamma}\,ds\\
\nonumber
&\leq&c\ep^{2\gamma\xi}|t'-t|^{1-2\alpha}
\\
\nonumber
&\leq& c \ep^{2\gamma\xi}
  \ep^{2(1/2-\alpha-\alpha(1/2-\alpha)+\alpha\delta') }
|t'-t|^{2\alpha(1/2-\alpha-\delta')}\\
\label{equt:14}
&\leq& c \ep^{2\hat p}|t'-t|^{2\alpha(1/2-\alpha-\delta')},
\end{eqnarray}
where the last inequality follows since 
\begin{eqnarray}
\label{equt:13}
 \hat p< 1/2-\alpha+\gamma\xi-\alpha(1/2-\alpha)+\alpha\delta'.
\end{eqnarray} 

As for $D_2(t)$, we again use Lemma~\ref{tubnd}, and also argue similarly to the proof of Lemma~\ref{lem:23_1}:

\begin{eqnarray*}
D_2(t)&=& \int_0^{t-\ep}  \left| p_{t-s}(x)- p_{t'-s}(x)\right|^2|\widetilde{X}(s,0)|^{2\gamma}\,ds\\
\nonumber
 &\leq&C(K)2^{2\gamma\xi N_{\xi}} \int_0^{t-\ep}  \left| ((t-s)^{-\alpha}-(t'-s)^{-\alpha})e^{-\frac{|x|^{1/\alpha}}{t-s}}\right|^2 (t-s)^{2\gamma\xi}\,ds
\\
\nonumber
&&\mbox{}+
 +C(K)2^{2\gamma\xi N_{\xi}}\int_0^{t-\ep} \left| (t'-s)^{-\alpha}(e^{-\frac{|x|^{1/\alpha}}{t-s}}-e^{-\frac{|x|^{1/\alpha}}{t'-s}})\right|^2
 (t-s)^{2\gamma\xi}\,ds
\\
\nonumber
&=:&D_{2,1}+D_{2,2}.
\end{eqnarray*}
Then we easily have
\begin{eqnarray*}
D_{2,1}
 &\leq&C(K)2^{2\gamma\xi N_{\xi}} \int_0^{t-\ep} ((t-s)^{-2\alpha-2}(t'-t)^{2} (t-s)^{2\gamma\xi}\,ds
\\
\nonumber
&\leq&C(K)2^{2\gamma\xi N_{\xi}} \ep^{((-2\alpha-1+2\gamma\xi)\wedge 0)-2\alpha \delta_2} (t'-t)^{2}\\
\nonumber
&\leq&C(K)2^{2\gamma\xi N_{\xi}} \ep^{((-2\alpha-1+2\gamma\xi)\wedge 0)-2\alpha \delta_2}\ep^{2-2\alpha(1/2-\alpha-\delta')} |t'-t|^{2\alpha(1/2-\alpha-\delta')}
\\
\nonumber
&=&C(K)2^{2\gamma\xi N_{\xi}} \ep^{2(((-\alpha+1/2+\gamma\xi)\wedge 1)-\alpha \delta_2-\alpha(1/2-\alpha)+\alpha\delta')} |t'-t|^{2\alpha(1/2-\alpha-\delta')}
\\
&=& C(K)2^{2\gamma\xi N_{\xi}} \ep^{2\hat p} (t'-t)^{\alpha(1-2\alpha-2\delta')},
\end{eqnarray*}
and 
\begin{eqnarray*}
D_{2,2}
 &\leq&C(K)2^{2\gamma\xi N_{\xi}} \int_0^{t-\ep} (t'-s)^{-2\alpha}\left|\frac{|x|^{1/\alpha}}{t-s}-\frac{|x|^{1/\alpha}}{t'-s}\right|^{2}
(t-s)^{2\gamma\xi}
 \,ds
\\
\nonumber
&\leq& C(K)2^{2\gamma\xi N_{\xi}} |x|^{2/\alpha}\int_0^{t-\ep} (t'-s)^{-2\alpha}(t-s)^{-4}|t'-t|^2
(t-s)^{2\gamma\xi}
 \,ds
\\
&\leq& C(K)2^{2\gamma\xi N_{\xi}} |x|^{2/\alpha} \ep^{-3-2\alpha+2\gamma\xi}|t'-t|^2\\
&\leq& C(K)2^{2\gamma\xi N_{\xi}} \ep^{2} \ep^{-3-2\alpha+2\gamma\xi}|t'-t|^{2\alpha(1/2-\alpha-\delta')}
 \ep^{2-2\alpha(1/2-\alpha-\delta')}\\
&=& C(K)2^{2\gamma\xi N_{\xi}} \ep^{2(1/2-\alpha+\gamma\xi-\alpha(1/2-\alpha)+\alpha\delta') }|t'-t|^{2\alpha(1/2-\alpha-\delta')}\\
&\leq& C(K)2^{2\gamma\xi N_{\xi}} \ep^{2\hat p}|t'-t|^{2\alpha(1/2-\alpha-\delta')},
\end{eqnarray*}
where the last inequality follows by~(\ref{equt:13}).

Combining the above bounds, we are done. 

\end{proof}

We can finally conclude that in (\ref{eq:H3}), $P_{3}=P_{4}=0$ if
\begin{equation}
\label{eq:Hoelderimp1}
C(K) \epsilon^{2\hat p}
2^{2  N_{\xi} \xi \gamma} 
< \epsilon^{2p}
\end{equation}
 For (\ref{eq:Hoelderimp1}) it is equivalent to show 
\[ C(K)<2^{2N(\hat p -p)-2N_{\xi}\gamma\xi}\]
and since $\hat p-p=\delta'-\delta_1-\delta_2>0$ we require
\begin{eqnarray}
\nonumber
 N&>& \left\lfloor\frac{2\gamma\xi N_{\xi}+\log C(K)}{2(\hat p -p)}\right\rfloor+1\\
\label{equt:21}
&=&\left\lfloor\frac{2\gamma\xi N_{\xi}+\log C(K)}{2(\delta'-\delta_1-\delta_2)}\right\rfloor+1.
\end{eqnarray} 
where $\lfloor\cdot\rfloor$ is the greatest integer function. Hence by~(\ref{equt:21}) we can choose the constant 
\begin{eqnarray}
\label{equt:22} 
c_{\ref{equt:22}}=c_{\ref{equt:22}}(K,\xi,\delta,\delta_1,\delta',\delta_2) 
\end{eqnarray}
such that for 
\begin{eqnarray*}
\label{equt:23} 
N\geq \left[c_{\ref{equt:22}}N_{\xi}\right]
\end{eqnarray*}
(\ref{eq:Hoelderimp1}) holds. 
 Note that the constant $c_{\ref{equt:22}}$ depends ultimately on $\xi,\xi_1$ and $K$. 
Hence (\ref{eq:H3}), (\ref{equt:19}), (\ref{equt:20}) imply that if 
\begin{eqnarray}
 \label{equt:24}
N_1(\omega,\xi,\xi_1\,,K)=N_{\xi}\vee\left[c_{\ref{equt:22}}N_{\xi}\right]
\end{eqnarray} 
then  for
$d((t,x),(t',y))\le 2^{-N\alpha}$, $t\le t'$,
\begin{eqnarray}
\nonumber
& &\IP\Big(|\tilde X(t,x)-\tilde X(t,y)|\ge |x-y|^{1/2-\alpha-\delta}2^{-Np},\ (t,x)\in
Z_{K,N,\xi}, N\ge N_1 \Big )\\
\nonumber
& &\qquad+\IP \Big(|\tilde X(t',x)-\tilde X(t,x)|\ge |t'-t|^{\alpha
(1/2-\alpha-\delta)}2^{-Np},
\ (t,x)\in Z_{K,N,\xi}, t'\le T_K, N\ge N_1 \Big)\\
\label{keybnd}
& &\le
c_{\ref{P1est}}(\exp\Bigl(-c'_{\ref{P1est}}|x-y|^{-\delta''}\Bigr)
+\exp\Bigl(-c'_{\ref{P1est}}|t'-t|^{-\alpha\delta''}\Bigr)).
\end{eqnarray}
\noindent
Now set
\begin {eqnarray*}
M_{n,N,K}&=& \max\{|\tilde X(j2^{-n},(z+1)2^{-\alpha n})-\tilde X(j2^{-n},z2^{-\alpha n})|\\
& &\qquad +|\tilde X((j+1)2^{-n},z2^{-\alpha n})-\tilde X(j2^{-n},z2^{-\alpha n})|:\\
& &\qquad |z|\le 2^{\alpha n}, (j+1)2^{- n}\le T_K, j\in \IZ_+, z\in\IZ, \\
& & (j2^{-n},z2^{-\alpha n})\in Z_{K,N,\xi}\}.
\end{eqnarray*}
(\ref{keybnd}) implies that if
\[A_N=\{\omega:\hbox{ for some }n\ge N,\ M_{n,N,K}\ge 2\cdot
2^{-n\alpha (1/2-\alpha-\delta)}2^{-Np},\ N\ge N_1\},
\]
then for some fixed constants $C,c_1,c_2>0$, 
\begin{eqnarray*}
\IP(\cup_{N'\ge N}A_{N'})&\le& C\sum_{N'=N}^\infty\sum_{n=N'}^\infty
K 2^{(\alpha+1)n} e^{-c_1 2^{n\delta''\alpha}}\\
&\le&  C K \eta_N,\\
\end{eqnarray*}
where $\eta_N =e^{-c_{2} 2^{N \delta''\alpha}}.$
Therefore $N_2(\omega)=\min\{N\in\IN:\omega\in A_{N'}^c\ \hbox{for all }N'\ge
N\}<\infty$ a.s. and in fact
\begin{eqnarray}
\label{N3bnd}
\IP(N_2> N)=\IP(\cup_{N'\ge N}A_{N'})\le  CK \eta_N.
\end{eqnarray}
Choose $m \in \IN$ with $m > 2/\alpha$ and assume
$N\ge (N_2+m)\vee (N_1+m)$.  Let $(t,x)\in Z_{K,N,\xi}$, $d((t',y),(t,x))\le
2^{-N\alpha}$, and $t'\le T_K$.  For $n\ge N$ let $t_n\in 2^{-n}\IZ_+$ and $x_{n}\in
2^{-\alpha n}\IZ$  be the unique points so that $t_n\le t<t_n+2^{-n},$
$x_{n}\le x<x_{n}+2^{-\alpha n}$
for $x\ge 0$ and $x_{n}-2^{-\alpha n}<x\le x_{n}$ if $x<0$.
Similarly define $t'_n$ and
$y_n$ with $(t',y)$ in place of $(t,x)$.  Choose $(\hat{t},\hat{x})$ as in the definition of
$Z_{K,N,\xi}$ (recall $(t,x)\in Z_{K,N,\xi}$).  If $n\ge N$, then
\begin{eqnarray*}
d((t'_n,y_n),(\hat{t},\hat{x}))&\le&
d((t'_n,y_n),(t',y))+d((t',y),(t,x))+d((t,x),(\hat{t},\hat{x}))\\
 &\le& |t'_n-t'|^{\alpha}+|y-y_n|+2^{-N\alpha}+2^{-N\alpha}\\
&<& 4\cdot 2^{-N\alpha}< 2^{2-N\alpha}\\\
&<& 2^{-\alpha(N-2/\alpha)}<2^{-\alpha(N-m)}.
\end{eqnarray*}
Therefore $(t'_n,y_n)\in Z_{K,N-m,\xi}$, and similarly (and slightly more simply)
$(t_n,x_n)\in Z_{K,N-m,\xi}$.  Our definitions imply
that $t_N$ and $t'_N$ are equal or adjacent in $2^{-N}\IZ_+$ and similarly for
the components of $x_N$
and $y_N$ in $2^{-N\alpha}\IZ_+.$  This, together with the continuity of $\tilde X$, the 
triangle inequality,
and our lower bound on $N$ (which shows $N-m\ge (N_2\vee N_1)$), implies
\begin{eqnarray*}
|\tilde X(t,x)-\tilde X(t',y)|&\le& |\tilde X(t_N,x_N)-\tilde
X(t'_N,y_N)|\\
& &\qquad+\sum_{n=N}^\infty|\tilde X(t_{n+1},x_{n+1})-\tilde X(t_n,x_n)|
+|\tilde X(t'_{n+1},y_{n+1})-\tilde X(t'_n,y_n)|\\
&\le& M_{N,N-m,K}+\sum_{n=N}^\infty 2M_{n+1,N-m,K}\\
&\le& C\sum_{n=N}^\infty 2\cdot2^{-n\alpha (1/2-\alpha-\delta)}2^{-(N-m)p}\\
&\le& c_0(p)2^{-N(\alpha (1/2-\alpha-\delta)+p)}\\
&\le& 2^{-N\xi_1}.
\end{eqnarray*}
The last line is valid for $N\ge N_3$ because $\alpha (1/2-\alpha-\delta)+p>\xi_1$ by
(\ref{equt:10}). Here
$N_3$ is deterministic and may depend on $p,\xi_1,\delta,c_0$ and hence ultimately on
$\xi,\xi_1$.  This proves the required result with
\[N_{\xi_1}(\omega)=\max(N_2(\omega)+m,N_{\xi(\omega)}+m,[c_{\ref{equt:22}}(\xi,\delta_{1}) N_{\xi}]+m, N_3).
\]
Now fix  $R'=1\vee c_{\ref{equt:22}}(\xi,\delta_{1})$ and $N(K)\equiv N_3$ (deterministic).
Then if 
 \mbox{$N\geq 2m \vee N(K)$}, (\ref{N3bnd})
implies that
\begin{eqnarray*}
\IP(N_{\xi_1}\ge N)&\le&\IP(N_2\ge N-m)+2 \IP(N_\xi\ge N(1-m/N)/R')\\
&\le&
CK \eta_{N-m}+2 \IP(N_{\xi}\ge N/R),
\end{eqnarray*}
for $R=2R'$. This
 gives the required probability bound (\ref{Nbnd}).
 
 \gdm
 
 \section{Smooth kernels}
 \label{sec:smoothkernels}
The strong uniqueness results stated earlier had $\alpha>0$. We did not try to extend those arguments to the cases $\alpha=0$, because in that case a much simpler argument will serve. We present that in this section. 
 
If $\alpha=0$, then the SIE~\eqref{equt:cat6} is simply the SDE $dX_t=\sigma(X_t)\,dt$, and the classical Yamada-Watanabe result gives strong uniqueness for $\gamma\in[\frac12,1]$. But one can ask about more general SIE's, with a smooth but non-constant kernel, for which the latter result does not apply directly. That is the content of the following result. Note that this is the only result that in this paper that applies when $\gamma$ actually $=\frac12$. 

\begin{proposition} 
\label{prop:smoothcase}
Suppose that $\kappa(s,t)$ is a deterministic smooth positive function of variables $s\le t$, that 
is bounded away from 0. Let $\sigma$ satisfy~(\ref{eq:sigma_Hold}) for some  $\gamma\in [\frac12, 1]$. Then strong uniqueness holds for the stochastic integral equation 
\begin{equation}
\label{eq:smoothSIE}
X_t=x_0+\int_0^t\kappa(s,t)\sigma(X_s)\,dB_s.
\end{equation}
\end{proposition}

\begin{proof}
Let $X^1_t$ and $X^2_t$ be solutions to~\eqref{eq:smoothSIE}. Set $Y^i_t=\int_0^t\sigma(X_s)\,dB_s$, so $X^i_t=x_0+\int_0^t \kappa(s,t)\,dY^i_s$. Therefore $dX^i_t=\kappa(t,t)\,dY^i_t + H^i_t\,dt$, where $H^i_t=\int_0^t\partial_2\kappa(s,t)\,dY^i_s$. 

Set $\tilde X_t=X^1_t-X^2_t$, $\tilde Y_t=Y^1_t-Y^2_t$, and $\tilde H_t=H^1_t-H^2_t$, so 
$d\tilde X_t=\kappa(t,t)\,d\tilde Y_t +\tilde H_t\,dt$.
In particular, for $\phi_n$ as in~\eqref{def:phi}, 
\begin{multline*}
\phi_n(\tilde X_t)=\int_0^t\phi_n'(\tilde X_s)\kappa(s,s)\,d\tilde Y_s+\int_0^t\phi_n'(\tilde X_s)\tilde H_s\,ds+\\
+\frac12\int_0^t\phi_n''(\tilde X_s)[\sigma(X^1_s)-\sigma(X^2_s)]^2\,ds.
\end{multline*}
Let $K>0$ and take $T_K$ to be the first time either $X^1_t$ or $X^2_t$ exceeds K. Recall that  $L$ is the H\"older constant for $\sigma$. Then the quadratic variation of the first term is bounded, so 
\begin{align*}
&\IE[\phi_n(\tilde X_{t\land T_K})]\\
&\qquad=\IE\Big[\int_0^t\phi_n'(\tilde X_s)1_{\{s<T_K\}}\tilde H_s\,ds
+\frac12\int_0^t\phi_n''(\tilde X_s)(\tilde X_s)1_{\{s<T_K\}}[\sigma(X^1_s)-\sigma(X^2_s)]^2\,ds\Big]\\
&\qquad\le \IE\Big[ \int_0^t |\phi_n'(\tilde X_s)\tilde H_s|1_{\{s<T_K\}}\,ds
+\frac{L^2}{2}\int_0^t|\phi_n''(\tilde X_s)|1_{\{s<T_K\}}|\tilde X_s|^{2\gamma}\,ds\Big]\\
&\qquad\le \int_0^t \IE[|\phi_n'(\tilde X_{s\land T_K})\tilde H_{s\land T_K}|]\,ds
+\frac{L^2}{2}\int_0^t\IE[|\phi_n''(\tilde X_{s\land T_K})||\tilde X_{s\land T_K}|^{2\gamma}]\,ds\equiv I_1^n+I_2^n.
\end{align*}
Then 
$$
I_2^n\le \frac{L^2}{n}\int_0^t \IE[|\tilde X_{s\land T_K}|^{2\gamma-1}]\le \frac{L^2t(2K)^{2\gamma-1}}{n}\to 0\quad\text{as $n\to\infty$.}
$$
Since
$$
\tilde H_t=\int_0^t\partial_2\kappa(s,t)\,d\tilde Y_s=\partial_2\kappa(t,t)\tilde Y_t-\int_0^t\tilde Y_s\partial_{21}\kappa(s,t)\,ds,
$$
we have
\begin{align*}
I_1^n&\le \int_0^t \IE[|\tilde H_{s\land T_K}|]\,ds\\
&\le \int_0^t |\partial_2\kappa(s,s)|\IE[|\tilde Y_{s\land T_K}|]\,ds
+ \int_0^t \int_0^s|\partial_{21}\kappa(q,s)|\IE[|\tilde Y_{q\land T_K}|]\,dq\,ds\\
&=\int_0^t \IE[|\tilde Y_{s\land T_K}|]\Big[|\partial_2\kappa(s,s)|+\int_s^t|\partial_{21}\kappa(s,q)|\,dq\Big]\,ds.
\end{align*}
Sending $n\to\infty$ gives that
\begin{equation}
\label{eq:I1estimate}
\IE[|\tilde X_{t\land T_K}|]\le \int_0^t \IE[|\tilde Y_{s\land T_K}|]\Big[|\partial_2\kappa(s,s)|+\int_s^t|\partial_{21}\kappa(s,q)|\,dq\Big]\,ds.
\end{equation}
Let $m_K(t)=\max_{s\le t}\IE[|\tilde Y_{s\land T_K}|]$. Since 
$$
\tilde X_t=\int_0^t\kappa(s,t)\,d\tilde Y_s=\kappa(t,t)\tilde Y_t-\int_0^t\partial_1\kappa(s,t)\tilde Y_s\,ds,
$$
we see that $|\kappa(t,t) \tilde Y_t|\le |\tilde X_t|+\int_0^t|\partial_1\kappa(s,t)\tilde Y_s|\,ds$. In combination with~\eqref{eq:I1estimate}, this shows that 
$$
\IE[|\tilde Y_{t\land T_K}|]\le \int_0^t \IE[|\tilde Y_{s\land T_K}|]\Big[\frac{|\partial_2\kappa(s,s)|+\int_s^t|\partial_{21}\kappa(s,q)\,dq
+|\partial_1\kappa(s,t)|}{|\kappa(t,t)|}\Big]\,ds.
$$
For any $t_0>0$, let $C(t_0)$ be the maximum of the above fraction, over $0\le s\le t\le t_0$. Therefore
\begin{equation}
\label{eq:Gronwall}
0\le m_K(t)\le C(t_0) \int_0^t m_K(s)\,ds
\end{equation}
for every $t\le t_0$. This is $\le C(t_0)tm_K(t)$, from which it follows that $m_K(t)=0$ for $t\in [0,\frac1{C(t_0)}]$. Applying~\eqref{eq:Gronwall} a second time now gives this for $t\in [0,\frac2{C(t_0)}]$. After finitely many iterations we have $m_K(t)=0$ on $[0,t_0]$, and since $t_0$ was arbitrary, in fact this holds for all $t\ge 0$. Sending $K\to\infty$ shows that for every $t$ we have $\tilde Y_t=0$ a.s., and therefore also $\tilde X_t=0$ a.s.
\end{proof}


\end{document}